\documentclass{elsarticle}
\usepackage{hyperref}
\usepackage[T1]{fontenc}
\usepackage{stmaryrd} 
\usepackage{bbold} 
\usepackage{amsmath, amssymb,amsthm}
\usepackage{cleveref}
\usepackage{enumerate}
\usepackage{xcolor}
\usepackage{graphicx}
\usepackage{dutchcal}
\usepackage{verbatim}

\newcommand{\cN}{\mathcal{N}}
\newcommand{\cG}{\mathcal{G}}
\newcommand{\cE}{\mathcal{E}}
\newcommand{\cV}{\mathcal{V}}
\newcommand{\cU}{\mathcal{U}}

\newcommand{\slp}{\mathcal{S}}
\newcommand{\dlp}{\mathcal{D}}
\newcommand{\slo}{S}
\newcommand{\dlo}{D}
\newcommand{\dloa}{D'}
\newcommand{\hso}{H}
\newcommand{\ifun}{\mathbb{1}} 
\newcommand{\vint}{\mathcal{V}^\circ}
\newcommand{\real}{\mathbb{R}}
\newcommand{\complex}{\mathbb{C}}
\newcommand{\comm}[1]{\llbracket#1\rrbracket}
\newcommand{\expect}{\mathbb{E}}
\newcommand{\proba}{\mathbb{P}}
\newcommand{\diag}{\text{diag}\,}

\newtheorem{theorem}{\bf Theorem}[section]
\newtheorem{lemma}{\bf Lemma}[section]
\newtheorem{corollary}{\bf Corollary}[section]
\newtheorem{remark}{\bf Remark}[section]

\newtheorem{definition}{\bf Definition}[section]

\begin{document}

\title{A potential theory on weighted graphs\tnoteref{t1}}

\tnotetext[t1]{TD and FGV were supported by National Science Foundation Grants
DMS-2008610 and DMS-2136198.}

\author[1]{Trent DeGiovanni\fnref{fn1}}

\author[2]{Fernando Guevara Vasquez\corref{cor1}}
\ead{fguevara@math.utah.edu}

\cortext[cor1]{Corresponding author}

\affiliation[1]{organization={Department of Mathematics,
Dartmouth College},
city={Hanover},
postcode={03755},
state={NH},
country={USA}}

\affiliation[2]{organization={Department of Mathematics,
University of Utah},
city={Salt Lake City},
postcode={84112},
state={UT},
country={USA}}

\fntext[fn1]{Formerly: Department of Mathematics, University of Utah,
UT, 84112, USA}

\begin{abstract} 
We present an analog to classic potential theory on weighted graphs.
With nodes partitioned into exterior, boundary and interior nodes and an
appropriate decomposition of the Laplacian, we define discrete analogues
to the trace operators, the single and double layer potential operators,
and the boundary layer operators. As in the continuum, these operators
can represent exterior or interior harmonic functions with different
boundary conditions. The formalism we introduce includes a discrete
Calder\'on calculus and brings some well known results from potential
theory to weighted graphs, e.g. on the spectrum of the
Neumann-Poincar\'e operator. We illustrate the formalism with a cloaking
strategy on weighted graphs which allows to hide an anomaly from the
perspective of electrical measurements made away from the anomaly. 
\end{abstract}


\maketitle

\section{Introduction}
\label{sec:intro}
There are several techniques for partial differential equations (PDE)
that can be put under the umbrella of ``potential theory''. The key
feature is that they can reduce a PDE in $d$ dimensions to an integral
equation on a $d-1$ dimensional boundary. The dimension reduction
enables efficient methods such as the boundary element method, which are
especially attractive to solve scattering problems in free space. Our
goal here is to introduce similar techniques to weighted graphs, with as
few assumptions as possible on the graph, so that many results from
continuum potential theory can be carried out with few modifications on
weighted graphs. Although this is not the first ``discrete potential
theory'' (see \cref{sec:rev} for similar techniques), we believe that
our operator to operator approach comes close to this lofty goal, as it
leads to a Calder\'on calculus that is essentially the same as in the
continuum. By ``Calder\'on calculus'' we mean the exact relations that
tie the different potential operators and boundary potential operators,
e.g. jump conditions, Calder\'on projector, Plemelj's symmetrization
principle etc. A brief review for the Laplace equation in the continuum
is included for convenience in \cref{sec:cont}.  From a practical point
of view, a discrete potential theory that mimics the continuum one, can
serve as a foundation for numerical methods for boundary integral
equations that respect the structure of the continuum problem, as for
example \cite{Dominguez:2014:FDC,Dominguez:2014:NFC,Betcke:2021:BEMPP}.  
A discrete approach to potential theory is also desirable from a
pedagogical perspective as an introduction to potential theory that
requires only an understanding of the weighted graph Laplacian and some
linear algebra.

\subsection{Continuum potential theory for the Laplace equation}
\label{sec:cont}
We recall some results from continuum potential theory (see e.g.
\cite{Colton:2013:IEM,Steinbach:2008:NAM}).  The notation of this
section is self-contained and we intentionally write the operators of
potential theory and their discrete analogues in the same way. Since our
only objective is to highlight similarities between the continuum and
discrete potential theories, we omit important details such as the
spaces on which the operators are defined and boundary regularity
considerations.

We work on an open bounded domain $\Omega^- \in \mathbb{R}^n$ with a
smooth boundary $\partial \Omega$ and denote by $\Omega^+ = \mathbb{R}^n
\setminus \overline{\Omega^-}$. We shall use the convention
were a superscript ``$-$'' means inside $\partial\Omega$ and superscript
``$+$'' means outside $\partial\Omega$. Let $G$ be the Green function
for the Laplace equation in $\mathbb{R}^n$, i.e. a function $G(x)$
satisfying $\Delta G= -\delta(x)$ and vanishing as $|x| \to \infty$,
where $\Delta = \partial_1^2 + \ldots + \partial_n^2$ is the Laplacian
and $\delta$ is the Dirac delta distribution.  The single $\slp$ and
double  $\dlp$ layer potential operators are linear operators mapping
sufficiently regular functions (sometimes referred to as boundary
densities) on $\partial \Omega$ to functions on $\mathbb{R}^n \setminus
\partial\Omega$ as follows:
\[
\begin{aligned}
(\slp(\varphi))(x) &= \int_{\partial \Omega} G(x-y)\varphi(y)dS
~\text{and}~\\
(\dlp(\varphi))(x) &= \int_{\partial \Omega} \nu(y) \cdot \nabla_y G (x-y) \varphi(y) dS,
\end{aligned}
\]
where $\nu$ is the outward-pointing normal vector to $\partial \Omega$.
These layer potential operators can be used to write the following
reproduction formulas or Green identities. If $(\Delta u)|_{\Omega^-} = 0$
then we have the interior reproduction formula
\begin{equation}\label{eqn:green3}
 \slp(\gamma_0^- u) -\dlp (\gamma_1^- u) = \begin{cases}
    0 & \text{if}~x \in \Omega^+, ~\text{and}\\
    u & \text{if}~x \in \Omega^-,
\end{cases}
\end{equation}
where $\gamma_0^+$ (resp. $\gamma_0^-$) is the exterior (resp. interior)
Dirichlet trace operator, i.e. for $x\in \partial\Omega$: $(\gamma_0^\pm
u)(x) = \lim_{\epsilon \to 0^+} u(x\pm\epsilon \nu(x))$. The exterior
(resp. interior) Neumann trace operator is $\gamma_1^+$ (resp.
$\gamma_1^-$) and is given for $x\in\partial\Omega$ by
\[
 (\gamma_1^\pm u)(x) = \lim_{\epsilon \to 0^+} \frac{u(x\pm \epsilon \nu(x)) - u(x)}{\epsilon}.
\]
Also if $(\Delta u)|_{\Omega^+}= 0$ and $u$ decays sufficiently fast as
$|x| \to \infty$, we have the exterior reproduction formula:
\begin{equation}\label{eqn:green4}
 -\slp(\gamma_0^+ u) +\dlp (\gamma_1^+ u) = \begin{cases}
    u & \text{if}~x \in \Omega^+,~\text{and} \\
    0 & \text{if}~xx \in \Omega^-.
\end{cases}
\end{equation}
The Calder\'on calculus involves the following linear operators that map
densities on the boundary to densities on the boundary:
\begin{enumerate}[i.]
 \setlength{\itemsep}{0pt}
 \item Single layer operator: $\slo = \{\gamma_0\} \slp$,
 \item Double layer operator: $\dlo = \{\gamma_0\} \dlp$,
 \item Adjoint double layer operator: $\dloa = \{\gamma_1\} \slp$,
 \item Hypersingular operator: $\hso = -\{\gamma_1\} \dlp$,
\end{enumerate}
where $\{\gamma_i\} = \frac{1}{2} (\gamma_i^+ + \gamma_i^-)$, $i=0,1$. The
operator $\dloa$ is also referred to as the \emph{Neumann-Poincar\'e operator}.
These four fundamental operators satisfy the following jump formulas:
\begin{equation}
    \begin{aligned}
    \gamma_0^+ \slp &= \gamma_0^- \slp  =
    \slo,\\
    \gamma_1^\pm \slp &= \mp\frac{I}{2} + D',\\
    \gamma_0^\pm \dlp &= \pm\frac{I}{2} + D,\\
    -\gamma_1^+ \dlp & = -\gamma_1^- \dlp = H,
    \end{aligned}
    \label{eq:jump:cont}
\end{equation}
where $I$ is the identity. Finally define $P_\pm = \frac{I}{2}
\mp C$  where $C$ is given by
\begin{equation}
    C = \begin{bmatrix} 
    -\dlo & \slo \\
    \hso & \dloa
    \end{bmatrix}.
\end{equation}
The operator $P_-$ (resp. $P_+)$ is known as the interior (resp.  exterior)
\emph{Calder\'on projector} and satisfies $P_\pm^2 = P_\pm$. The relevance
of the Calder\'on projectors is that they give boundary integral equations
satisfied by the Cauchy data for $u^-$ harmonic on $\Omega^-$ (for the
interior Calder\'on projector) and $u^+$ harmonic on $\Omega^+$ (for the exterior Calder\'on
projector), namely
\begin{equation}
    P_\pm 
    \begin{bmatrix}
        \gamma_0^\pm u^\pm \\ \gamma_1^\pm u^\pm
    \end{bmatrix}
    =
    \begin{bmatrix}
        \gamma_0^\pm u^\pm \\ \gamma_1^\pm u^\pm
    \end{bmatrix}.
\end{equation}
As we shall see, the discrete potential formalism that we introduce has analog
properties to the ones presented here. Further analogue properties are presented
later in their own context.

\subsection{Previous work on discrete potential theory}
\label{sec:rev}
Discrete potential theory has been formulated previously in different
contexts. For example Green functions, Green identities and boundary
reproduction formulas are well known for weighted graphs, albeit
without single and double layer potentials, jump relations or a discrete
Calder\'on calculus \cite{Bendito:2007:PTB,Bendito:2008:BVP}. A discrete
potential theory has also been derived on infinite lattices
\cite{Bhamidipati:2021:FIM} using boundary algebraic equations
\cite{Gunnar:2009:BAE}, which are an analogue of the boundary integral
equations in the continuum, and that define boundary operators mapping
boundary densities to the rest of the lattice by basing the densities on
edges. Unfortunately, this leads to a formulation of the discrete
potential theory which is inconsistent with certain results in the
continuum (namely jump relations, see e.g. \cite{Colton:2013:IEM}).
These infinite lattice methods are closely related to the method of
difference potentials, a discretization approach for continuum problems
on regular grids that (among other features) can get higher order
accuracy boundary reproductions of fields by implicitly representing
the boundary on a discretization nodes that are in a neighborhood of the
boundary \cite{Ryabenkii:2001:MDP}. Difference potential methods have
also been adapted to have an explicit formulation of the double-layer
potential \cite{Gurlebeck:2002:FDP}. A key difference between the
lattice potential theory and the method of difference potentials is that
in the former, the boundary is defined on the lattice nodes. This allows
for consistency with the continuum potential theory in terms of e.g.
jump relations. We also point out a discrete Calder\'on calculus that is
derived as a discretization of the continuum Helmholtz and elasticity
equations
\cite{Dominguez:2014:FDC,Dominguez:2014:NFC,Dominguez:2015:FDC}, where
the jump relations are satisfied by the discretization.  The formulation
we present here considers boundary potentials as node-valued quantities
and applies directly to weighted graphs.

\subsection{Contents}
\label{sec:contents}
We start in \cref{sec:harm} by reviewing harmonic and Green functions on
weighted graphs. The potential theory on weighted graphs that we propose is
derived in \cref{sec:dpt}. In \cref{sec:appl} we show how certain
properties that are well known in the continuum hold in on weighted
graphs. In particular we show how to reduce the exterior/interior
Dirichlet or Neumann boundary value problems to integral equations.
An application to active cloaking in graphs is presented in
\cref{sec:cloaking}.

\section{Harmonic functions on graphs}
\label{sec:harm}
We fix some notation in \cref{sec:setup}. Then we review the graph
Laplacian in \cref{sec:gl} and the graph Green function in
\cref{sec:graph_green}.

\subsection{Graph setup}
\label{sec:setup}
Let $\cG = (\cV, \cE)$ be a finite connected graph with nodes (or
vertices) $\cV$ and edges
$\cE \subset \cV \times \cV$. The graphs we work with are non-oriented and have
no self edges, so that edge $\{x,y\} \in E$ is identical to edge $\{y,x\}$ and
there are no edges of the form $\{x,x\}$. We partition the nodes into four
non-intersecting and non-empty sets $B$, $\Omega^+$,  $\partial \Omega$ and
$\Omega^-$. We further assume that out of these four node subsets, the only
node subset pairs that are connected are $\{B,\Omega^+\}$,
$\{\!\{\Omega^+,\partial\Omega\}\!\}$ and
$\{\!\{\partial\Omega,\Omega^-\}\!\}$. We say that a
node subset pair $\{\!\{X,Y\}\!\}$, where  $X,Y \subset \cV$ is connected if there
are $x\in X$ and $y\in Y$ such that $\{x,y\} \in \cE$. As in the continuum, we
call $\Omega^+$ the exterior, $\partial\Omega$ the boundary and $\Omega^-$ the
interior. We shall see later in \cref{sec:graph_green} that the set $B$ is needed to
guarantee a unique definition for the graph Green function.
We give an example graph satisfying these connectivity restrictions
in \cref{fig:setup}. 

\begin{remark}
\label{rem:thick:boundary}
In continuum potential theory, the boundary $\partial\Omega$ is defined
topologically as $\overline{\Omega^-}\setminus \Omega^-$. For graphs, if we are given $\Omega^-
\subset \cV$, we may define its boundary similarly by finding all the nodes
that are adjacent to a node in $\Omega^-$ but not in $\Omega^-$. This is an
example of a ``thin'' boundary and it can satisfy the connectivity requirements
we impose. However we also allow quite arbitrary ``thick'' boundaries.
\end{remark}

\begin{figure}[h]
    \centering
    \includegraphics[width=0.5\textwidth]{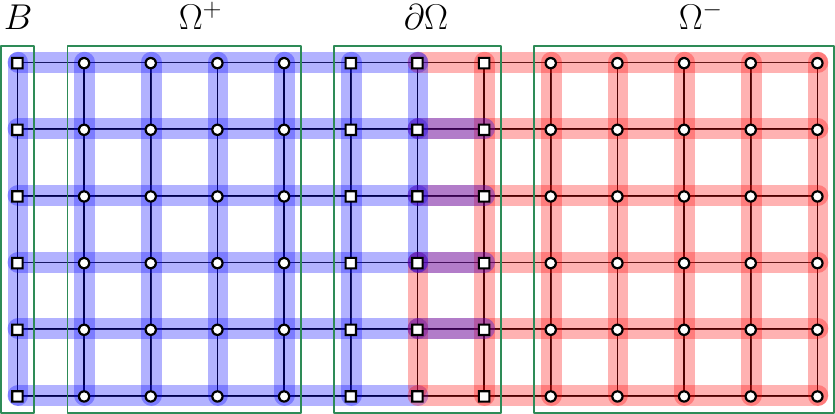}
    \caption{An example of a graph where the only node subset pairs
    that are connected are $\{\!\{B,\Omega^+\}\!\}$,
    $\{\!\{\Omega^+,\partial\Omega\}\!\}$ and
    $\{\!\{\partial\Omega,\Omega^-\}\!\}$. The colors of the edges
    represent a possible edge partition of the unity. Blue represents
    $p^+ = 1$, red is $p^-=1$ and violet a mixture of the two with
    $p^+ + p^- = 1$.  }
    \label{fig:setup}
\end{figure}

\subsection{Graph Laplacian}
\label{sec:gl}
Let us fix a priori the ordering of the nodes $\cV$ and edges $\cE$.
For a set $X$ we write its cardinality with $|X|$.  The \emph{discrete
gradient} $\nabla$ is a $|\cE|\times |\cV|$ matrix such that for any
nodal quantity $u \in \real^{|\cV|}$, $(\nabla u)_k = u_i - u_j$ where
$i$ and $j$ are the node indices associated with the $k-$th edge,
given in an order that is known a priori. The ordering of the nodes
on each edge that is used in the definition
of the discrete gradient is irrelevant, as long as it
remains fixed. We sometimes refer to and use nodal or edge based quantities as
``functions''. A \emph{conductivity} is a positive weight assigned to each
edge, or in the predetermined ordering of $\cE$, some $\sigma \in
(0,\infty)^{|\cE|}$. The \emph{weighted graph Laplacian} associated with
a conductivity $\sigma \in (0,\infty)^{|\cE|}$ is the $|\cV|  \times
|\cV|$ matrix defined by $L(\sigma) = \nabla^T \diag(\sigma) \nabla$,
where $\diag(\sigma)$ is the diagonal matrix with $\sigma$ in its
diagonal. A node based quantity $u \in
\real^{|\cV|}$ is said to be \emph{$\sigma-$harmonic} on some $X \subset
\cV$ if $(L(\sigma) u)_X = 0$. We say that $u \in \real^{|\cV|}$ solves
the \emph{Dirichlet problem} with Dirichlet boundary condition $u|_B =
f$ if 
\begin{equation}
 \begin{aligned}
 (L(\sigma) u)|_{\vint} &= 0,\\
 u|_B &= f,
 \end{aligned}
 \label{eq:dir}
\end{equation} 
where $\vint \equiv \cV \setminus B$. In other words, $u$ is $\sigma-$harmonic
on $\vint$ with prescribed boundary data $f$ on $B$. A physical interpretation of
\eqref{eq:dir} is that the weighted graph represents a resistor network, where
the $k-$th edge is a resistor with resistance $\sigma_k^{-1}$ (the conductance
is the reciprocal of resistance) and the resistors are connected at the
nodes. Then for $i \in \vint$, $u_i$ is the voltage of the $i-$th node 
resulting from imposing voltages $f_j$ at all nodes $j \in B$. A well
established result is that if $\cG$ is connected, then the Dirichlet problem admits
a unique solution for any boundary voltage $f$, see e.g. \cite{Chung:1997:SGT}. 
  
\subsection{Graph Green function}\label{sec:graph_green}
A basic notion in potential theory is the Green function.
On a weighted graph, we may represent a Green function by a $|\cV|
\times |\cV|$ matrix $G$ that can be defined by blocks as 
 \begin{equation}
     \label{eqn:greendef}
     G = \begin{bmatrix} G[B,B] & G[B,\vint]\\
     G[\vint,B] & G[\vint,\vint]
     \end{bmatrix}
     =\begin{bmatrix}
         0 & 0\\
         0 & L[\vint,\vint]^{-1}
     \end{bmatrix},
 \end{equation}
where for clarity we have omitted the dependency of the graph Laplacian $L$ on
$\sigma$. Note that the inverse in \eqref{eqn:greendef} exists because $\cG$ is
a connected graph. If $G_j$ is the column of
$G$ associated with node $j$, then either $j \in B$ so that $G_j = 0$; or $j
\in \vint$, and $G_j$ solves the Dirichlet problem
\begin{equation}
	\begin{aligned}
		(L(\sigma)G_j)|_{\vint} &= e_j,\\
	    G_j|_{B} & = 0,
	\end{aligned}
\end{equation}
where $e_j$ is the $j-$th canonical basis vector of $\real^{|\cV|}$.
Thus a physical interpretation for $G_j$ is the voltage in a resistor
network resulting from injecting a unit current at the $j-$th node, $j
\in \vint$, while imposing zero voltages at the boundary $B$.

Alternatively, we may write $G$ using restriction matrices. For some $X\subset
\cV$, we consider the $|X| \times |\cV|$ matrix $R_X$ that restricts a vector
defined on all the nodes to $X$, i.e.  $R_Xu = u|_X$. The Green function can
then be written as 
\begin{equation}
    G = R_{\vint}^T L[\vint,\vint]^{-1} R_{\vint}.
\end{equation}
The Moore-Penrose pseudoinverse of $G$ (see e.g. \cite{Golub:2013:MC})
is given by 
\begin{equation}
	G^\dagger = R_{\vint}^T L[\vint,\vint] R_{\vint},
\end{equation}
and has the following property
\begin{equation}
        G^\dagger G = G G^\dagger = R_{\vint}^T R_{\vint}.
\end{equation}
Since the Laplacian is symmetric, the Green function is also symmetric.
This can be thought of as a reciprocity property, i.e. the voltage
measured at the $i-$th node resulting from a current source at the
$j-$th node is the same as the voltage at the $i-$th node resulting
from a current source at the $i-$th node.  As in the continuum (see
e.g.  \cite{Evans:2010:PDE}), the Green function can be used to solve
Dirichlet problems with non-zero source terms. For example if $\phi \in
\real^{\cV}$ satisfies $\phi|_B = 0$, then by construction $u = G \phi$
solves the Dirichlet problem with source term $\phi$
\begin{equation}
\label{eq:greenrw}
    \begin{aligned}
        (L(\sigma) u)|_{\vint} &= \phi,\\
        u|_B &= 0.
    \end{aligned}
\end{equation}

\begin{remark}
\label{rem:B}
In the continuum, no set $B$ is needed, but $\Omega^+$ is infinite and this
Dirichlet boundary condition is generally replaced by a radiation boundary
condition at infinity. We believe the set $B$ will not be needed when we expand this
formalism to lattices, but that is the subject of current research.
\end{remark}

\section{A discrete potential theory}
\label{sec:dpt}
Our discrete potential theory is based on a decomposition of the
Laplacian into an exterior and an interior Laplacian
(\cref{sec:decomp}). This allows us to define in \cref{sec:trace}
discrete counterparts to the Dirichlet and Neumann trace operators. The
discrete single and double layer potential operators are introduced in
\cref{sec:slp-dlp}. The single and double layer potential operators
together with appropriate traces allow us to reproduce $\sigma-$harmonic
functions in different sets of nodes, using the reproduction formulas
in \cref{sec:repro}. Taking boundary traces of the single and double
layer potential operators yields discrete analogues of continuum
boundary layer operators and jump formulas (\cref{sec:blo}). Finally in
\cref{sec:calderon} we show that the boundary layer operators we define
are tied together by Calder\'on projectors. In other words taking the
interior (resp. exterior) Cauchy data of a function $\sigma-$harmonic in the
interior (resp. exterior) is enough to reproduce this function in the
interior (resp. exterior).

\subsection{A decomposition of the Laplacian}
\label{sec:decomp}
Our construction relies on decomposing the Laplacian into two graph Laplacians,
one associated with the interior and another with the exterior. To define the
possible decompositions, let us first partition the edges into the edges
$\cE^+$ having at least one node in $B \cup \Omega^+$, the edges $\cE^-$
having at least one node in $\Omega^-$ and $\cE^{\partial\Omega}$, the edges
having two nodes in $\partial\Omega$. We then define an \emph{edge partition
of unity} of the form $p^+, p^- \in [0,1]^{|\cE|}$, with $p^+ + p^- = 1$,
$p^+|_{\cE^+} = 1$ and $p^-|_{\cE^-}=1$. An example is given in \cref{fig:setup}.
We are left with considerable freedom on the choice of $p^\pm$ on
$\cE^{\partial\Omega}$. All boundary operator definitions involving fluxes (or
currents) depend on the choice of $p^\pm$.  Intuitively, the edge partition of
unity assigns a fraction of the flux through each edge connecting two
nodes 
in $\partial\Omega$ to either the fluxes flowing from the interior to the
boundary or from the boundary to the exterior (or the exterior and interior
Neumann trace operators that we define later in \cref{sec:trace}).

The decomposition of the Laplacian is based on the conductivities
$\sigma^\pm = p^\pm \sigma$, where the multiplication is understood
componentwise. By linearity of $L(\sigma)$ we can decompose the Laplacian as follows
\begin{equation}
 \label{eq:decomp}
 L(\sigma) =  L(\sigma^+) + L(\sigma^-),
\end{equation}
where we call $L(\sigma^+)$ the \emph{exterior graph Laplacian} and
$L(\sigma^-)$ the \emph{interior graph Laplacian}. Since the edge
weights $\sigma^\pm$ can be zero, we cannot call them conductivities,
however we can simply delete any edges where they vanish to see that
$L(\sigma^\pm)$ are indeed weighted graph Laplacians. The decomposition
is illustrated in \cref{fig:cartoon}.

\begin{figure}
 \centering
 \includegraphics[width=0.6\textwidth]{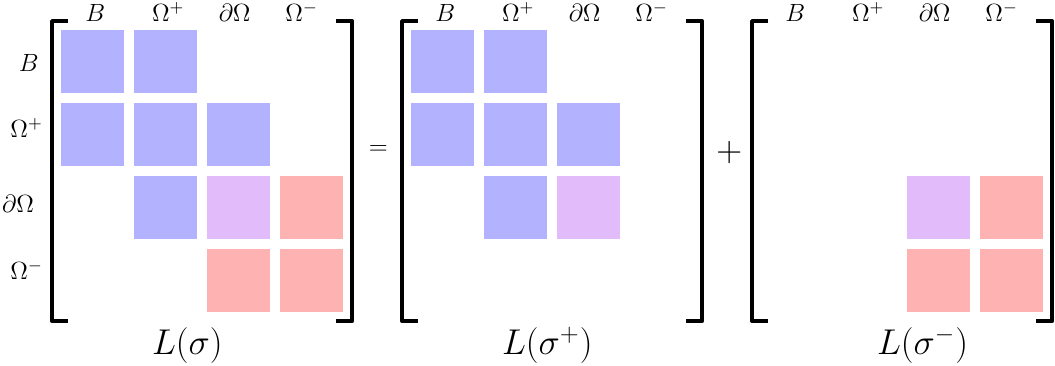}
 \caption{An illustration of the decomposition of the graph Laplacian
 $L(\sigma)$ into exterior $L(\sigma^+)$ and interior $L(\sigma^-)$ Laplacians.
 The blue blocks correspond to blocks that are identical in $L(\sigma)$ and
 $L(\sigma^+)$. The red blocks indicate identical blocks in $L(\sigma)$ and
 $L(\sigma^-)$. The purple $\partial\Omega,\partial\Omega$ block depends on the
 choice of the edge partition of unity $p^\pm$ and is different in $L(\sigma)$,
 $L(\sigma^+)$ and $L(\sigma^-)$.}
 \label{fig:cartoon}
\end{figure}

\subsection{Trace operators}
\label{sec:trace}
The other ingredients we need are discrete analogues of the Dirichlet
and Neumann trace operators from the continuum, see e.g.
\cite{Evans:2010:PDE}. We define the discrete trace operators by
\begin{equation}\label{eqn:trace_operator}
\begin{aligned}
    \gamma_0 &=  R_{\partial\Omega}, &&\text{(Dirichlet trace operator) and}\\
    \gamma_1^{\pm} &= \mp R_{\partial\Omega} L(\sigma^{\pm}). &&\text{(Neumann trace operators)}
\end{aligned}
\end{equation}
In the continuum we distinguish between $\gamma_0^+$, the limit from the
outside and $\gamma_0^-$, the limit from the inside. In a weighted graph,
the Dirichlet trace operator $\gamma_0$ is simply restricting a node
valued quantity  to $\partial \Omega$. It does not make sense to
distinguish $\gamma_0^\pm$ for graphs, as they are both equal. However,
two different Neumann trace operators $\gamma_1^\pm$ are
needed for graphs. The exterior one $\gamma_1^+$ is associated
with the exterior Laplacian and the interior one $\gamma_1^-$ with the
interior Laplacian. To go further, recall that for any $u \in
\real^{|\cV|}$, $L(\sigma) u$ can be understood as the net currents at
all the nodes that are needed to maintain the voltage $u$. Thus the
\emph{interior Neumann trace operator} $\gamma_1^-$ calculates the
currents flowing out of the interior $\Omega^-$ through the boundary
$\partial \Omega$.  Likewise, the \emph{exterior Neumann trace operator}
calculates the currents going from $\partial\Omega$ to $\Omega^+$. The
particular signs used in the Neumann trace operator definitions come from
orienting both fluxes from the interior to the exterior. 

We emphasize that the notion of flux depends on the choice of the
conductivity and partition of unity of the edges.  Intuitively, the
partition of unity $p^+,p^-$ controls the fraction of the currents
internal to $\partial\Omega$ that goes to the interior or exterior
Neumann trace operators.  Notice also that the difference of fluxes
gives, up to a sign, the net currents at the $\partial\Omega$ nodes,
a quantity that is independent of the partition.  Indeed for any $u \in
\real^{|\cV|}$, we have
\begin{equation}
 \label{eq:flux}
 \gamma_1^+u - \gamma_1^-u = - (L(\sigma) u)|_{\partial\Omega}.
\end{equation}
A consequence of \eqref{eq:flux} is that a function $u$ is
$\sigma-$harmonic on $\partial\Omega$ if and only if $\gamma_1^+ u =
\gamma_1^- u$. In other words, there are no source or sinks of currents on
$\partial\Omega$ so the currents flowing out of and into the boundary
are identical.

\subsection{Layer potential operators}
\label{sec:slp-dlp}
We are now ready to define single and double-layer potentials which are
analogues of the continuum ones (see e.g. \cite{Colton:2013:IEM}), as
they define quantities in the whole space ($\cV$ for graphs) from quantities defined on the boundary
$\partial\Omega$.
\begin{definition}\label{def:layer_pots}
The graph layer potential operators are the $|\cV| \times
|\partial\Omega|$ matrices given by
\begin{align}
  \slp &= G\gamma_0^T, & \text{(single layer potential)}\label{eq:slp}\\
  \dlp^\pm &= G[\gamma_1^\pm]^T. & \text{(exterior/interior double layer
  potential)}\label{eq:dlp}
\end{align}
Here $\dlp^+$ is the double layer potential associated with the exterior $\Omega^+$
and similarly $\dlp^-$ is associated with the interior $\Omega^-$. 
\end{definition}
The analogy between the operators of \cref{def:layer_pots} and the
continuum operators can be taken further. Indeed if $\psi \in
\real^{\partial\Omega}$, then $\slp \psi$ is a linear combination of
``monopoles'' (the columns of $\slp$) weighted by the entries of $\psi$.
Physically this corresponds to injecting currents at the nodes in
$\partial\Omega$. Similarly, $\dlp^\pm \psi$ is a linear combination of
``dipoles'' (the columns of $\dlp^{\pm}$) weighted by $\psi$. We
emphasize that the choice of $\dlp^+$ or $\dlp^-$ impacts which set of
nodes is used to inject currents. These sets can be found by looking at
the non-zero rows of $\gamma_1^+$ and $\gamma_1^-$ and are clearly
different: for $\dlp^{\pm}$ we need to inject
currents at the nodes $\cN(\partial \Omega)
\setminus \Omega^\mp$, where $\cN(X)$ is the closed neighborhood of $X
\subset \cV$, i.e. $X$ and all the nodes adjacent to $X$.  

As in the continuum, the layer potential operators $\slp$ and $\dlp^\pm$
take a boundary density and give a $\sigma-$harmonic functions outside
of $\partial\Omega$. Indeed for any $\phi \in \real^{|\partial\Omega|}$
direct calculations reveal that 
\begin{equation}
  (L(\sigma) \slp \phi)|_{\vint} =  (\gamma_0^T \phi)|_{\vint}
  ~\text{and}~
  (L(\sigma) \dlp^\pm \phi)|_{\vint} = ((\gamma_1^\pm)^T \phi)|_{\vint}.
\end{equation}
In particular, this implies that $(L(\sigma) \slp
\phi)|_{\Omega^+\cup\Omega^-} =  0$ and $(L(\sigma) \dlp^\pm
\phi)|_{\Omega^+\cup\Omega^-} = 0$. However these functions are not
$\sigma-$harmonic on $\partial\Omega$.

\subsection{Reproduction formulas}
\label{sec:repro}
We can now introduce reproduction formulas similar to the ones in the
continuum, with one caveat being that we have four formulas rather than
the usual two, depending on the double layer potential we use (either
exterior $\dlp^+$ or interior $\dlp^-$) or depending on whether we want
to reproduce a function that is $\sigma-$harmonic (i.e. with zero net
currents) inside $\Omega^+$ or $\Omega^-$.
\begin{theorem}\label{thm:intrep}
The following reproduction formulas hold:
    \begin{align}
    \label{eq:intrep+}
    \slp(\gamma_1^{+}u)-\dlp^{+}(\gamma_0 u) &= u \ifun_{\partial\Omega \cup \Omega^-},
    &&\text{for}~u~\text{s.t.}~Lu|_{\Omega^- \cup \partial \Omega} = 0,\\
    \label{eq:intrep-}
    \slp(\gamma_1^{-}u)-\dlp^{-}(\gamma_0 u) &= u \ifun_{\Omega^-},
    &&\text{for}~u~\text{s.t.}~Lu|_{\Omega^-} = 0,\\
    \label{eq:extrep+}
    -\slp(\gamma_1^{+}u)+\dlp^{+}(\gamma_0 u) &= u \ifun_{B \cup \Omega^+},
    &&\text{for}~u~\text{s.t.}~Lu|_{\Omega^+} = 0,~u|_B=0,\\
    -\slp(\gamma_1^{-}u)+\dlp^{-}(\gamma_0 u) &= u \ifun_{B \cup \Omega^+ \cup \partial\Omega},
    \label{eq:extrep-}
    &&\text{for}~u~\text{s.t.}~Lu|_{\Omega^+ \cup \partial\Omega} = 0,~u|_B=0.
    \end{align}
    Here we used $\ifun_X \in \real^{|\cV|}$ for the indicator function of a
    set $X \subset \cV$ and multiplication is understood componentwise.
    The dependency of $L$ on $\sigma$ is omitted for clarity.
\end{theorem}
\begin{proof}
    To shorten notation, we use $L^\pm \equiv L(\sigma^\pm)$. Using the
    commutator of two square matrices $\comm{D,C}=DC-CD$ we get
    \[
      \slp\gamma_1^\pm  - \dlp^\pm\gamma_0  = 
      \mp GR_{\partial\Omega}^TR_{\partial\Omega}L^{\pm} \pm GL^{\pm}R_{\partial\Omega}^TR_{\partial\Omega} 
      = \pm G \comm{ L^\pm,  R_{\partial\Omega}^TR_{\partial\Omega}}.
    \]
    We start by proving \eqref{eq:intrep+}. Since
    $(u\ifun_{\partial\Omega\cup\Omega^-})|_B =
    0$, it is enough to verify that 
    \begin{equation}
        \comm{ L^+, R_{\partial\Omega}^TR_{\partial\Omega} } u =
	G^\dagger (u\ifun_{\partial\Omega\cup\Omega^-}).
    \end{equation}
    An explicit calculation gives 
    \begin{equation}
     \comm{L^+, R_{\partial\Omega}^TR_{\partial\Omega}} u =
           \begin{bmatrix}  
           0 & 0 & 0 & 0 \\
           0 & 0 & L[\Omega^+,\partial\Omega] & 0\\
           0& -L[\partial\Omega,\Omega^+] & 0 & 0\\
           0 & 0 & 0 & 0
     \end{bmatrix}u
     = \begin{bmatrix}
         0\\
         L[\Omega^+,\partial\Omega]u|_{\partial\Omega}\\
         -L[\partial\Omega,\Omega^+]u|_{\Omega^+}\\
         0
     \end{bmatrix},
    \end{equation}
    where we used the node ordering $B \cup \Omega^+ \cup \partial \Omega \cup
    \Omega^-$. Likewise we get
    \begin{equation}
    G^\dagger (u\ifun_{\partial\Omega\cup\Omega^-}) =
    \begin{bmatrix}
    0\\
    L[\Omega^+,\partial\Omega]u|_{\partial\Omega}\\
    L[\partial\Omega,\partial\Omega] u|_{\partial\Omega} +
    L[\partial\Omega,\Omega^-] u|_{\Omega^-}\\
    L[\Omega^-,\partial\Omega]u|_{\partial\Omega} + L[\Omega^-,\Omega^-]
    u|_{\Omega^-}
    \end{bmatrix}
    =  
    \begin{bmatrix}
         0\\
         L[\Omega^+,\partial\Omega]u|_{\partial\Omega}\\
         -L[\partial\Omega,\Omega^+]u|_{\Omega^+}\\
         0
    \end{bmatrix},
    \end{equation}
    where the last equality comes from the assumption that $(L u)
    |_{\partial\Omega\cup \Omega^-} = 0$. This proves the reproduction 
    formula \eqref{eq:intrep+}.

    We now prove \eqref{eq:intrep-}. Similarly, it is sufficient to
    verify that
    \begin{equation}
        \comm{ L^-, R_{\partial\Omega}^TR_{\partial\Omega} } u =
	-G^\dagger (u\ifun_{\Omega^-}).
    \end{equation}
    Direct calculations together with the assumption $L u|_{\Omega^-} = 0$ reveal that indeed
    \begin{equation}
        \comm{ L^-,  R_{\partial\Omega}^TR_{\partial\Omega} } u
        = \begin{bmatrix}
            0\\0\\
            -L[\partial\Omega,\Omega^-]u|_{\Omega^-}\\
            L[\Omega,\partial\Omega^-]u|_{\partial\Omega}
        \end{bmatrix}
        = -G^\dagger (u\ifun_{\Omega^-}).
    \end{equation}
    We can proceed similarly to verify the remaining reproduction
    formulas. For \eqref{eq:extrep+}, since we have $Lu|_{\Omega^+}=0$
    and $u|_B=0$, direct calculations give $ \comm{ L^+,
    R_{\partial\Omega}^TR_{\partial\Omega} } u = - G^\dagger
    (u\ifun_{\Omega^+})$. For \eqref{eq:extrep-}, we have 
    $Lu|_{\Omega^+\cup \partial\Omega}=0$ and $u|_B=0$, which leads to $\comm{
    L^-, R_{\partial\Omega}^TR_{\partial\Omega} } u =  G^\dagger
    (u\ifun_{\Omega^+\cup\partial\Omega})$.
\end{proof}

\begin{remark}
Recall that $u$ is $\sigma-$harmonic on $\partial\Omega$ if and only if $\gamma_1^+ u =
\gamma_1^-u$. Therefore the reproduction formulas \eqref{eq:intrep+}
and \eqref{eq:extrep-}, which both assume  $u$ is $\sigma-$harmonic on
$\partial\Omega$, are valid for both Neumann traces $\gamma_1^\pm u$.
\end{remark}

\subsection{Boundary layer operators and jump formulas}
\label{sec:blo}
The boundary layer operators are defined by taking Dirichlet or Neumann
traces of the layer potential operators, which are defined as follows.
\begin{definition}\label{def:blo}
    The boundary layer operators are the following 
    $|\partial\Omega| \times |\partial\Omega|$ matrices:
    \begin{enumerate}[i.]
     \setlength{\itemsep}{0pt}
     \item Single layer operator: $\slo = \gamma_0 \slp$.
     \item Double layer operator: $\dlo = \frac{1}{2} \gamma_0 (\dlp^+ + \dlp^-)$.
     \item Adjoint Double layer operator: $\dloa = \frac{1}{2} (\gamma_1^+ + \gamma_1^-) \slp = \dlo^T$.
     \item Hypersingular operator: $\hso = -\gamma_1^+ \dlp^- = -\gamma_1^- \dlp^+$.
    \end{enumerate}
\end{definition}

The equality in \cref{def:blo}(iv) is not obvious and is proved in
\cref{thm:jump}, as part of other jump relations across the boundary
$\partial\Omega$ that are satisfied by the boundary layer potentials. As
in the continuum the boundary layer operators appear when studying the
traces of single and double layer potentials.
\begin{theorem}\label{thm:jump}
Let $I$ be the identity matrix. The following jump relations hold
\begin{align}
    \label{eq:slpjump}
    \gamma_1^\pm \slp &= \mp \frac{I}{2} + D',\\
    \label{eq:dlpjump}
    \gamma_0\dlp^\pm &= \mp \frac{I}{2} + D,~\text{and}\\
    \label{eq:dlpcont}
    \gamma_1^+\dlp^- &= \gamma_1^-\dlp^+.
\end{align}
\end{theorem}
\begin{proof}
    First notice that in the node ordering $B \cup \vint$ we have
    \[
    LG = \begin{bmatrix}
        0 & L[B,\vint]L[\vint,\vint]^{-1} \\ 0 & I
    \end{bmatrix}
    ~\text{and}~
    GL = \begin{bmatrix}
        0 & 0 \\ L[\vint,\vint]^{-1} L[\vint,B] & I
    \end{bmatrix}.
    \]
    To prove \eqref{eq:slpjump}, we can see that
    \[
    \begin{aligned}
    \gamma_1^\pm \slp &= \mp R_{\partial\Omega} L^\pm G R_{\partial\Omega}^T = \mp R_{\partial\Omega} \frac{L^+ +L^-}{2} G R_{\partial\Omega}^T - R_{\partial\Omega} \frac{L^+ - L^-}{2} G R_{\partial\Omega}^T\\
    &= \mp \frac{1}{2} R_{\partial\Omega} L G R_{\partial\Omega}^T  + D'
     = \mp \frac{I}{2}  + D'.
    \end{aligned}
    \]
    Similarly, to prove \eqref{eq:dlpjump}, we notice that
    \[
    \begin{aligned}
    \gamma_0 \dlp^\pm &= \mp R_{\partial\Omega} G L^\pm R_{\partial\Omega}^T = \mp R_{\partial\Omega} G \frac{L^+ +L^-}{2}  R_{\partial\Omega}^T - R_{\partial\Omega} G \frac{L^+ - L^-}{2}  R_{\partial\Omega}^T\\
    &= \mp \frac{1}{2} R_{\partial\Omega}  G L R_{\partial\Omega}^T  + D'
     = \mp \frac{I}{2}  + D.
    \end{aligned}
    \]
    Finally we get \eqref{eq:dlpcont} from
    \[
    \begin{aligned}
    \gamma_1^+\dlp^- &= R_{\partial\Omega} L^+ G L^- R_{\partial\Omega}^T\\
    &= R_{\partial\Omega} (L-L^-) G (L-L^+) R_{\partial\Omega}^T\\
    &= R_{\partial\Omega} L^- G L^+ R_{\partial\Omega}^T 
    + R_{\partial\Omega} (L^+ GL - LG L^+ )R_{\partial\Omega}^T\\
    &=\gamma_1^-\dlp^+.
    \end{aligned}
    \]
\end{proof}
\begin{remark}
The jump relations we obtain in \cref{thm:jump} are identical to those
in the continuum \eqref{eq:jump:cont}, with two notable differences.
First the continuum jump relation for the values of the double layer
potential across the boundary is $\gamma_1^\pm \slp = \mp \frac{I}{2} +
D'$, so there is a minus sign difference in the identity term compared
to the discrete identity \eqref{eq:dlpjump}. Second the hypersingular
operator is in the continuum $H = \gamma_1^+ \dlp = \gamma_1^- \dlp$. In
weighted graphs, we use two different double layer potential
operators, one for the exterior and one for the interior, see
\eqref{eq:dlpcont}.
\end{remark}

\subsection{Calder\'on projectors}
\label{sec:calderon}
We now present a discrete analogue of the interior and exterior
Calder\'on projectors. The derivation is inspired from that of the
continuum, see e.g. \cite{Steinbach:2008:NAM}, and proceeds by taking
the Dirichlet and Neumann traces of one of the reproduction formulas in
\cref{thm:intrep} and finding a system of boundary integral equations
that is satisfied by the Cauchy data, i.e. the Dirichlet and Neumann
data together. In weighted graphs, only the reproduction formulas
\eqref{eq:intrep+} and \eqref{eq:extrep-} that include $\partial\Omega$
can be used.
\begin{theorem}\label{thm:calderon}
Define the following matrix using the boundary operators of \cref{def:blo}
\begin{equation}
 C = \begin{bmatrix}
     -\dlo & \slo \\
     \hso & \dloa
 \end{bmatrix}.
\end{equation}
Then the matrices $P^\pm = \frac{I}{2} \mp C$ are projection matrices,
i.e. $(P^\pm)^2 = P^\pm$.  As in the continuum, we call $P^+$ the
\emph{exterior discrete Calder\'on projector} and $P^-$ the
\emph{interior discrete Calder\'on projector}.
\end{theorem}
\begin{proof}
We start by proving that $P^-$ is a projector. Consider two arbitrary
$\psi,\phi \in \real^{|\partial\Omega|}$. The function $u = \slp \psi -
\dlp^+ \phi$ is $\sigma-$harmonic on $\Omega^-$ and we can calculate its
Dirichlet and interior Neumann traces using the jump relations
\eqref{eq:slpjump},
\eqref{eq:dlpjump} and 
\cref{def:blo}  we obtain
\[
 \begin{aligned}
  \gamma_0 u &= \gamma_0 \slp \psi - \gamma_1^- \dlp^+ \phi 
  = \slo \psi + (\frac{I}{2} - D )\phi, \\
  \gamma_1^- u &=  \gamma_0 \slp \psi - \gamma_1^- \dlp^+ \phi
  = (\frac{I}{2} + \dloa)\psi + \hso \phi,
 \end{aligned}
\]
which in system form is:
\begin{equation}
  \begin{bmatrix}
     \gamma_0 u\\ \gamma_1^- u
 \end{bmatrix}
 =
 \begin{bmatrix} \frac{I}{2}-D & S \\
 H & \frac{I}{2} + D'
 \end{bmatrix}
 \begin{bmatrix}
     \phi \\ \psi
 \end{bmatrix}
 =
 P^- 
  \begin{bmatrix}
     \phi \\ \psi
 \end{bmatrix}.
\end{equation}

To proceed, we would like to apply the interior reproduction formula
\eqref{eq:intrep+} directly to $u$, but we encounter an issue: $u$ is
$\sigma-$harmonic on $\Omega^-$ but not on $\partial\Omega$. Indeed both single
and double layer potentials are fields generated by sources supported on
a neighborhood of $\partial\Omega$. However, if there exists a
$\sigma-$harmonic function $\widetilde{u}$ with
$\widetilde{u}|_{\partial\Omega \cup \Omega^-} = u|_{\partial\Omega \cup
\Omega^-}$, then necessarily $\gamma_1^+ \widetilde{u} = \gamma_1^-
\widetilde{u}$. But by construction we also have $\gamma_1^-
\widetilde{u} = \gamma_1^- u$, hence $\gamma_1^+
\widetilde{u} = \gamma_1^- u$. Intuitively, defining $\gamma_1^+
\widetilde{u} = \gamma_1^- u$ corresponds to injecting currents at the
nodes in $\partial\Omega$ and adjacent nodes in $\Omega^+$ so as
to ensure that there are no sources or sinks of currents in
$\partial\Omega$. Crucially, this does not necessitate defining
$\widetilde{u}$ outside of $\partial\Omega \cup \Omega^-$, only that the
fluxes $\gamma_1^\pm \widetilde{u}$ are well defined and equal to
$\gamma_1^- u$. Taking
the Dirichlet and interior Neumann traces of \eqref{eq:intrep+} applied
to $\widetilde{u}$, using the appropriate jump relations and putting in
system form we get:
\begin{equation}
  \begin{bmatrix}
     \gamma_0 \widetilde{u}\\ \gamma_1^- \widetilde{u}
 \end{bmatrix}
 =
 \begin{bmatrix} \frac{I}{2}-D & S \\
 H & \frac{I}{2} + D'
 \end{bmatrix}
 \begin{bmatrix}
     \gamma_0 \widetilde{u}\\ \gamma_1^+ \widetilde{u}
 \end{bmatrix}.
\end{equation}
As discussed earlier we have $\gamma_0 \widetilde{u} = \gamma_0 u$ and
$\gamma_1^\pm \widetilde{u} = \gamma_1^- u$, hence
\begin{equation}
  \begin{bmatrix}
     \gamma_0 u\\ \gamma_1^- u
 \end{bmatrix}
 =
 \begin{bmatrix} \frac{I}{2}-D & S \\
 H & \frac{I}{2} + D'
 \end{bmatrix}
 \begin{bmatrix}
     \gamma_0 u\\ \gamma_1^- u
 \end{bmatrix}
 =
 (P^-)^2
 \begin{bmatrix} \phi\\\psi \end{bmatrix}.
\end{equation}
Since $\phi,\psi \in \real^{|\partial\Omega|}$ are arbitrary, we have
established that $(P^-)^2 = P^-$, and therefore $P^-$ is a projector.  

That $P^+$ is a projector can be obtained in a similar manner by taking
the Dirichlet and exterior Neumann traces of some $v = -\slp \psi +
\dlp^- \phi$, where $\phi,\psi \in \real^{|\partial\Omega|}$ are
arbitrary. Clearly $v$ is $\sigma-$harmonic on $\Omega^+$ and we can
define $\widetilde{v}$ that is identical to $v$ on $B \cup \Omega^+ \cup
\partial \Omega$. If we further impose that $\gamma_1^- \widetilde{v} =
\gamma_1^+ \widetilde{v} = \gamma_1^+ v$,  we see that $\widetilde{v}$
is also $\sigma-$harmonic on $\partial\Omega$. To conclude, we take the
Dirichlet and exterior Neumann traces of the exterior reproduction 
formula \eqref{eq:extrep-}. An alternate (shorter) proof is
to use that $P^- = \frac{I}{2} + C$ is a projector and hence 
\begin{equation}
 (P^-)^2(\frac{I}{2}+C)^2 = \frac{I}{4} + C + C^2 = \frac{I}{2} + C =
 P^-.
\end{equation}
In other words, $C$ must satisfy $C^2=\frac{I}{4}$. Then clearly we have
that: 
\begin{equation}
 (P^+)^2 = (\frac{I}{2}-C)^2 = \frac{I}{4} - C + C^2 = \frac{I}{2} - C =
 P^+.
\end{equation}
\end{proof}

A consequence of the Calder\'on projector property is that a
$\sigma-$harmonic $u$ on $\Omega^-$ is uniquely determined by its Cauchy
data $\gamma_0 u$, $\gamma_1^- u$. Similary, a $\sigma-$harmonic $v$ on
$\Omega^+$ with $v|_B = 0$ is uniquely determined by its Cauchy data
$\gamma_0 v$, $\gamma_1^+ v$.
The following result follows by Considering $C^2=I/4$ block by block.
\begin{corollary}
The boundary integral operators satisfy the following.
\begin{align}
 \slo \hso &= -\dlo^2 + \frac{I}{4},\label{eq:SH}\\
 \hso \slo &= -(\dloa)^2 + \frac{I}{4},\label{eq:HS}\\
 \slo \dloa &= \dlo \slo,\label{eq:plemelj}\\
 \dloa \hso &= \hso \dlo.\label{eq:DH}
\end{align}
\end{corollary}
These identities are similar to those in the continuum, see e.g.
\cite[Corollary 6.19]{Steinbach:2008:NAM}. The relation \eqref{eq:plemelj} can be used
to symmetrize the double layer potential $\dlo$ by using an inner
product defined using the single layer potential $\slo$ and is called 
Plemelj symmetrization principle. This is discussed in more detail in
\cref{sec:npo}.

\section{Applications}
\label{sec:appl}
We take the exterior and interior Dirichlet problem in \cref{sec:dir}
and reformulate them as a boundary integral equation, i.e. as a $|\partial\Omega|
\times |\partial\Omega|$ linear system. The Dirichlet to
Neumann map is presented in \cref{sec:dtn}, and it can be represented as
in the continuum in terms of the boundary linear operators. Moreover, we
show that our definition of Dirichlet to Neumann map corresponds to that
in a graph with boundary. Then in \cref{sec:npo} we prove that some
results on the spectrum of the Neumann-Poincar\'e operator that are well
known in the continuum also hold in weighted graphs. Finally, the
exterior and interior Neumann problems are transformed into a boundary
integral equation in \cref{sec:neu}.

\subsection{Dirichlet problem}
\label{sec:dir}
We start with the following result.
\begin{lemma}
 \label{lem:slo}
 The single layer boundary operator $\slo$ is an invertible symmetric positive definite
 matrix with non-negative entries.
\end{lemma}
\begin{proof}
 Notice that $\slo = \gamma_0 G \gamma_0^T =
 (L[\vint,\vint]^{-1})[\partial\Omega,\partial\Omega]$.  Because of the
 connectivity of $\cG$, $L[\vint,\vint]$ is symmetric positive definite.
 Hence $\slo$ is a principal major of a symmetric positive matrix and
 thus must also be symmetric positive definite. For the
 sign property, notice that $L[\vint,\vint]$ is a non-singular
 $M-$matrix so that all the entries of its inverse are non-negative.
\end{proof}
We point out that $\slo$ being symmetric positive definite is analogous
to the ellipticity of its continuum version (see e.g.
\cite[Theorem 6.22]{Steinbach:2008:NAM}).
The \emph{interior Dirichlet problem} consists in
finding $u$ such that
\begin{equation} 
 \begin{aligned}
  (L(\sigma^-) u)|_{\Omega^-} &= 0, \\
  u|_{\partial\Omega} &= f,\\
 \end{aligned}
 \label{eq:dir:int}
\end{equation}
for some $f \in \real^{|\partial\Omega|}$. We emphasize that
\eqref{eq:dir:int} does not restrict in any way $u$ at the $B \cup
\Omega^+$ nodes. If we use the ansatz $u = \slp \phi$, then $f =
\gamma_0 u =\gamma_0 \slp \phi = \slo \phi$. Hence $\phi = \slo^{-1} f$.
The solution to \eqref{eq:dir:int} is unique on $\partial\Omega \cup
\Omega^-$ and can be written as $u = \slp \slo^{-1} f$.  We can proceed in a similar fashion with the
\emph{exterior Dirichlet problem} which is to find $v$ such that
\begin{equation}
 \begin{aligned}
  (L(\sigma^+) v)|_{\Omega^+} &= 0, \\
  v|_{\partial\Omega} &= f,\\
  v|_{B} &=0.
 \end{aligned}
 \label{eq:dir:ext}
\end{equation}
Notice that \eqref{eq:dir:ext} does not restrict in any way the values
of $v$ at $\Omega^-$.  Using the ansatz $v =\slp \phi$, we see that the
solution to \eqref{eq:dir:ext} is unique on $B \cup \Omega^+ \cup
\partial\Omega$ and is given by $v = \slp \slo^{-1} f$.  Thus both the
interior and exterior Dirichlet problems can be reduced to a discrete
boundary integral equation, i.e. to solving a linear system of size
$|\partial\Omega| \times |\partial\Omega|$ instead of a system of a
larger size on either $B \cup \Omega^+ \cup \partial\Omega$ or
$\partial\Omega \cup \Omega ^-$.

\begin{remark}
In both the interior \eqref{eq:dir:int} and exterior \eqref{eq:dir:ext}
Dirichlet problems we may replace $L(\sigma^\pm)$ by $L(\sigma)$. Indeed
for any $u,v \in \real^{|\cV|}$ we have $(L(\sigma^-) u)|_{\Omega^-} =
(L(\sigma) u)|_{\Omega^-}$ and  $(L(\sigma^+) v)|_{\Omega^+} =
(L(\sigma) v)|_{\Omega^+}$. See also \cref{fig:cartoon}.
\end{remark}

\subsection{Dirichlet to Neumann map}
\label{sec:dtn}
We define the interior (resp. exterior) Dirichlet to Neumann map at the
boundary $\partial\Omega$ as the linear map (matrix)
$\Lambda_{\sigma^-}$ (resp. $\Lambda_{\sigma^+}$) that maps the
Dirichlet trace $\gamma_0 u$ to the Neumann trace $\gamma_1^\pm u$ for a
function $u$ that is $\sigma-$harmonic on $\Omega^-$ (resp. on
$\Omega^+$ and with $u|_B = 0$). This is also known as the
Steklov-Poincar\'e or voltages-to-current, since it is the matrix that
maps voltages at the boundary $\partial\Omega$ to the currents at
$\partial\Omega$. The following result shows that there are several ways of
writing this mapping using the boundary layer operators. Such identities
are well known in the continuum (see e.g.  \cite[\S
6.6.3]{Steinbach:2008:NAM}).  
\begin{theorem}
 \label{thm:dtn:id}
 The interior $\Lambda_\sigma^-$ and exterior $\Lambda_\sigma^+$ Dirichlet
 to Neumann maps admit the following representations:
 \begin{equation}
 \begin{aligned}
 \Lambda_\sigma^{\pm} &= (\mp\frac{I}{2} +
 \dloa) \slo^{-1} = \slo^{-1} (\mp \frac{I}{2} + \dlo)\\
 &= \mp H \mp ( \mp \frac{I}{2} + \dloa ) \slo^{-1} ( \mp \frac{I}{2} + \dlo
 ).
 \end{aligned}
 \label{eq:dtn:def}
\end{equation}
\end{theorem}
\begin{proof}
Since the solution to both interior and exterior Dirichlet
problems with boundary condition $\gamma_0 u = f$ is $u = \slp \slo^{-1}
f$, the interior/exterior Dirichlet to Neumann maps can be found by
taking interior/exterior Neumann traces, namely
\begin{equation}
 \Lambda_\sigma^{\pm} = \gamma_1^\pm \slp \slo^{-1} = (\mp\frac{I}{2} +
 D') S^{-1},
\end{equation}
which proves the first equality. To show the remaining equalities, let
us focus first on the equalities involving the interior Dirichlet to
Neumann map $\Lambda_\sigma^-$. Recall that if $u$ is
$\sigma-$harmonic on $\Omega^-$ then
\begin{equation}
 \begin{bmatrix} \gamma_0 u \\ \gamma_1^- u \end{bmatrix}
 = P^-
 \begin{bmatrix} \gamma_0 u \\ \gamma_1^- u \end{bmatrix}.
 \label{eq:cauchy:int}
\end{equation}
By using the first row in \eqref{eq:cauchy:int} we can solve for the
Neumann data $\gamma_1^- u$ in terms of the Dirichlet data $\gamma_0u$
to obtain $\gamma_1^- u = \slo^{-1} (\frac{I}{2} \mp \dlo) \gamma_0 u$,
which establishes the second equality in \eqref{eq:dtn:def}. The third
equality in \eqref{eq:dtn:def} comes from the second row of
\eqref{eq:cauchy:int} by using $\gamma_1^- u = \slo^{-1} (\frac{I}{2}
\mp \dlo) \gamma_0 u$ and solving for $\gamma_1^- u$ in terms of
$\gamma_0 u$. The corresponding equalities for the exterior Dirichlet to
Neumann map $\Lambda_\sigma^+$ come in a similar fashion from
considering the Cauchy data $\gamma_0 v$, $\gamma_1^+ v$ for a
$\sigma-$harmonic $v$ on $\Omega^+$ with $v|_B = 0$.
\end{proof}
A note of caution: the Neumann traces $\gamma_1^\pm$ and the Dirichlet
to Neumann map depend on the partition of unity $p^\pm$ of the edges.
We now show that the Dirichlet to Neumann map definitions agree with the
traditional Schur complement definitions on either the subgraph $\cG^-$
of $\cG$ with nodes $\partial\Omega \cup \Omega^-$ or the subgraph
$\cG^+$ with nodes $B \cup \Omega^+ \cup \partial\Omega$.

\begin{theorem}
\label{thm:dtn}
The Dirichlet to Neumann map notion \eqref{eq:dtn:def} agrees with the
Schur complement definition (see e.g. \cite{Curtis:1998:CPG}) on the
graphs $\cG^\pm$ with edge weights $\sigma^\pm$. In other words the
following identity holds with $L^\pm \equiv L(\sigma^\pm)$:
\begin{equation}
 \label{eq:dtn:equiv}
 \Lambda_\sigma^\pm = \mp ( L^\pm[\partial\Omega,\partial\Omega] -
 L^\pm[\partial\Omega,\Omega^\pm] L^\pm[\Omega^\pm,\Omega^\pm]^{-1}
 L^\pm[\Omega^\pm,\partial\Omega]).
\end{equation}
\end{theorem}
\begin{proof}
The solution to the interior and exterior Dirichlet problems with boundary data $f$ is $u =
\slp S^{-1} f$. If each of $\cG^\pm$ is considered by itself, the
solution may be written 
 (see e.g. \cite{Curtis:1998:CPG}):
\[
 u|_{\partial\Omega \cup \Omega^\pm} = \begin{bmatrix} I \\
 -L[\Omega^\pm,\Omega^\pm]^{-1}
 L[\Omega^\pm,\partial\Omega] 
 \end{bmatrix} f.
\]
Now from the definition
\eqref{eqn:trace_operator} of the Neumann trace operators we see that
\[
\begin{aligned}
 \gamma_1^- &= R_{\partial\Omega} L^- = \left[ 0 ,
 0 , L^-[\partial\Omega,\partial\Omega] ,
 L^-[\partial\Omega,\Omega^-] \right],~\text{and}\\
 \gamma_1^+ &= -R_{\partial\Omega} L^+ = \left[ 0 ,
 L^+[\partial\Omega,\Omega^+],
 L^-[\partial\Omega,\partial\Omega], 0 \right].
\end{aligned}
\]
The claimed expression comes from calculating $\gamma_1^\pm u$ in terms
of $f$. Notice that for the exterior (resp. interior) Dirichlet problem,
only the graph $\cG^+$ (resp. $\cG^-$) is used. Therefore to define
$\Lambda_\sigma^\pm$, the $\cG^-$ graph (resp. $\cG^+$ graph) and the
values of $u$ therein do not need to be defined as they play no role.
\end{proof}
\subsection{Neumann-Poincar\'e operator}
\label{sec:npo}
The adjoint double boundary layer operator $\dloa$ is also called
\emph{Neumann-Poincar\'e operator}. The key fact to study this operator
is the Plemelj symmetrization principle \eqref{eq:plemelj}, as it states that $\dloa$
and $\dlo$ are similar matrices and therefore have the same spectrum.
Let us start by studying whether $\pm \frac{1}{2}$ are eigenvalues of $\dloa$.
\begin{lemma}
\label{lem:neu:bie:int}
Assume the subgraph $\cG^-$ of $\cG$ induced by the nodes
$\partial\Omega \cup \Omega^-$ is connected\footnote{By this we mean
that to determine connectivity, we only consider the edges $e$ where
$\sigma^-(e) >0$.} with the edge weights $\sigma^-$. Then the
Neumann-Poincar\'e operator $\dloa$ has eigenvalue $-1/2$ with
associated eigenspace spanned by $\slo^{-1} 1$.
\end{lemma}
\begin{proof}
Let $\phi$ be such that $(\dloa + \frac{I}{2})\phi = 0$. Consider $u =
\slp \phi$. By the jump relations we have $\gamma_1^- u = \gamma_1^- \slp \phi
= (\dloa+\frac{I}{2})\phi = 0$. The function $u$ is $\sigma-$harmonic on
$\Omega^-$ so $0 = (\gamma_0 u)^T (\gamma_1^- u) = u^T L(\sigma^-) u$.
By the assumption on the connectivity of $\cG^-$ with the weight
$\sigma^-$, we see that $u$ must be constant on $\partial\Omega \cup
\Omega^-$, say $u = C$. Taking Dirichlet trace $\gamma_0 u = \slo \phi =
C$. Therefore the eigenspace of $\dloa$ associated with eigenvalue
$-1/2$ is spanned by the single vector $\slo^{-1} 1$.  
\end{proof}

A similar proof can be applied to the exterior, however in this case our
choice of homogeneous Dirichlet condition at $B$ prevents finding an
eigenvector. This can be seen as follows.
\begin{lemma}
\label{lem:neu:bie:ext}
Assume the subgraph $\cG^+$ of $\cG$ induced by the nodes $B \cup
\Omega^+ \cup \partial\Omega$ is connected with the edge weights
$\sigma^+$. Then the operator $\dloa - \frac{I}{2}$ is invertible.
\end{lemma}
\begin{proof}
Let $\phi$ be such that $(\dloa - \frac{I}{2})\phi = 0$ and define $u =
\slp \phi$. By the jump relations: $\gamma_1^+ u = \gamma_1^+ \slp \phi
u = (\dloa - \frac{I}{2})\phi = 0$. Now $u$ is $\sigma-$harmonic on
$\Omega^+$ with $u|_B=0$. Therefore $0 = (\gamma_0 u)^T (\gamma_1^+ u) =
- u^T L(\sigma^+) u$. Since $\cG^+$ is connected with edge weights
$\sigma^+$ we conclude that $u$ is constant on
$B\cup\Omega^+\cup\partial\Omega$. Using that $u|_B = 0$ we can see 
that $u=0$. Taking the Dirichlet trace $\gamma_0 u = \slo \phi = 0$ and
using that $\slo$ is invertible we get $\phi=0$, which proves the
desired result.
\end{proof}
The Plemelj  symmetrization principle \eqref{eq:plemelj} also implies the following.
\begin{lemma}
 The Neumann-Poincar\'e operator $\dloa$ has a real spectrum.
\end{lemma}
\begin{proof}
The Plemelj identity \eqref{eq:plemelj} means that $\dloa$ is symmetric
with respect to the $\complex^{|\partial\Omega|}$ inner-product defined
for some $u,v \in \complex^{|\partial\Omega|}$ by $\langle
u, v\rangle_{\slo} = u^* \slo v$.  Indeed we have $\langle \dloa u,v\rangle_{\slo} =
(\dloa u)^* \slo v = u^* \dlo \slo v = u^* \slo \dloa v = \langle u,
\dloa v \rangle_{\slo}$, which shows the desired result.
\end{proof}
Finally we have the following result on the spectrum of $\dloa$, which
is analogous to a similar result in the continuum see e.g. the review
\cite{Ando:2021:SAN}.
\begin{theorem}
If $\lambda$ is an eigenvalue of $\dloa$ and $\lambda \neq \pm
\frac{1}{2}$ then $\lambda \in (-\frac{1}{2},\frac{1}{2})$.
\end{theorem}
\begin{proof}
Let $\lambda$ be an eigenvalue of $\dloa$ with $\lambda \neq \pm
\frac{1}{2}$ and $\phi \neq 0$ be a corresponding eigenvector with
$\dloa \phi = \lambda \phi$. Let $u =
\slp \phi$. We obtain by the jump relations that $\gamma_1^\pm u = (\mp
\frac{I}{2} + D') \phi = (\mp \frac{1}{2} + \lambda) \phi$. With this in
mind we notice that
\[
 (\gamma_0 u)^T (\gamma_1^\pm u) = (\mp \frac{1}{2} + \lambda) (\gamma_0
 u)^T \phi.
\]
By multiplying both sides by $\lambda \pm \frac{1}{2}$:
\[ 
 (\lambda \pm \frac{1}{2}) (\gamma_0 u)^T (\gamma_1^\pm u) = (\lambda^2
 - \frac{1}{4}) (\gamma_0 u)^T \phi.
\]
Since $u$ is $\sigma-$harmonic on $\Omega^-$ we see that $(\gamma_0 u)^T
(\gamma_1^- u) = u^T L(\sigma^-) u$. Similarly since $u$ is
$\sigma-$harmonic on $\Omega^+$ and $u|_B=0$, we get that $(\gamma_0
u)^T (\gamma_1^+u) = -u^TL(\sigma^+) u$. Putting it all together:
\[
 \mp (\lambda \pm \frac{1}{2})  u^TL(\sigma^\pm) u = (\lambda^2 -
 \frac{1}{4}) \phi^T \slo \phi.
\]
Since $\lambda \neq - \frac{1}{2}$ we get that $-u^T L(\sigma^+)u =
(\lambda - \frac{1}{2}) \phi^T \slo \phi$. Now $\phi \neq 0$ and $\slo$
symmetric positive definite  imply that $\phi^T \slo \phi > 0$. But
$L(\sigma^+)$ is symmetric positive semidefinite so $u^T L(\sigma^+) u
\geq 0$. We deduce that $\lambda < \frac{1}{2}$. Similarly since
$\lambda \neq \frac{1}{2}$ we see that $u^T L(\sigma^-) u = (\lambda +
\frac{1}{2}) \phi^T \slo \phi$. With $u^T L(\sigma^-) u \geq 0$, we can
also deduce that $\lambda > -\frac{1}{2}$. Thus we get the desired
result $\lambda \in (-\frac{1}{2},\frac{1}{2})$.
\end{proof}

\subsection{Neumann problem}
\label{sec:neu}
Consider the \emph{interior Neumann problem}:
\begin{equation}
\begin{aligned}
 (L(\sigma^-) u) |_{\Omega^-} &=0,\\
 \gamma_1^- u &= g,\\
\end{aligned}
\label{eq:neu:int}
\end{equation} 
for some $g \in \real^{|\partial\Omega|}$.  The following is a
discrete analogue of a basic fact for the continuum Neumann problem.
\begin{theorem}
 \label{thm:neumann:int}
 Assume the subgraph $\cG^-$ induced by the nodes $\partial\Omega
 \cup \Omega^-$ and with weight $\sigma^-$ is connected. Then the interior Neumann problem admits a unique
 solution up to a constant when the boundary condition $g \in
 \real^{|\partial\Omega|}$ satisfies the compatibility condition $1^T g
 = 0$. Moreover, we can replace $L(\sigma)$ by $L(\sigma^-)$ in
 \eqref{eq:neu:int}.
\end{theorem}
\begin{proof}
Notice that if $u$ is a solution to
\eqref{eq:neu:int} then so is $u+C$ for any constant $C$. Indeed,
for any constant $C$ we have $L(\sigma) C = 0$ and $\gamma_1^- C = 0$.
Now suppose there are two solutions $u_1$ and $u_2$ to \eqref{eq:neu:int}, then
$v = u_1 - u_2$ solves \eqref{eq:neu:int} with Neumann boundary
condition $g=0$. Consider the non-negative quantity 
\[
 v^T L(\sigma^-) v = v^T \nabla^T \diag(\sigma^-) \nabla v = \sum_{e \in
 \cE} \sigma^-(e)  ((\nabla v)(e))^2,
\]
which is the energy needed to maintain the potential $v$ on the subgraph
$\cG^-$ with conductivity $\sigma^-$. Since
$(L(\sigma) v)|_{\Omega^-} = (L(\sigma^-) v)|_{\Omega^-} = 0$, we can
see that:
\[
 v^T L(\sigma^-) v = \begin{bmatrix} v|_{\partial\Omega}\\ v|_{\Omega^-}
 \end{bmatrix}^T \begin{bmatrix} L^-[\partial\Omega,\partial\Omega] v|_{\partial\Omega}+
 L^-[\partial\Omega,\Omega^-] v|_{\Omega^-}\\ 0 \end{bmatrix} =
 v|_{\partial\Omega}^T \gamma_1^- v = 0.
\]
We conclude that $(\nabla v)(e) = 0$ for all edges $e$ for which
$\sigma^-(e) > 0$. Since the graph $\cG^-$ with edge weights $\sigma^-$
is connected, we conclude that $v$ is constant. Thus any two solutions
to \eqref{eq:neu:int} differ by a constant.

We now focus on the compatibility condition. Since constants are in the
nullspace of any graph Laplacian, we have $1^T
L(\sigma^-) u=0$. If in
addition
$u$ is a solution to \eqref{eq:neu:int} we see that
\[
 0 = 1^T L(\sigma^-) u = 1^T (L(\sigma^-) u)|_{\partial\Omega}  =
 1^T \gamma_1^- u = 1^T g,
\]
which is the compatibility condition.
\end{proof}

We can now solve \eqref{eq:neu:int} by taking the ansatz $u = \slp \phi$
for some $\phi \in \real^{|\partial\Omega|}$. Imposing the boundary
condition and using the jump relations we get
\begin{equation}
 g = \gamma_1^- u = \gamma_1^- \slp \phi = (\frac{I}{2} + \dloa) \phi.
 \label{eq:neu:bie}
\end{equation}
Now recall from \cref{lem:neu:bie:int} that $\text{null}(\frac{I}{2} +
\dloa) = \text{span}\,\{\slo^{-1} 1\}$. Using the Plemelj symmetrization
principle \eqref{eq:plemelj} we see that $\text{null}\,(\frac{I}{2} +
\dlo) = \text{span}\,\{1\}$. By the fundamental theorem of linear
algebra $\text{range}\,(\frac{I}{2} + \dloa) = \{1\}^\perp$. Hence the
compatibility condition $1^T g$ ensures that \eqref{eq:neu:bie} admits a
unique solution $\phi$.

We shift gears to the \emph{exterior Neumann problem}
\begin{equation}
 \begin{aligned}
  (L(\sigma^+)u)|_{\Omega^+} &=0,\\
  \gamma_1^+ u &= g,\\
  u|_B &= 0.
 \end{aligned}
\label{eq:neu:ext}
\end{equation}
By using the ansatz $v = \slp \phi$, imposing the boundary condition at
$\partial\Omega$ and the jump relations we see that
\begin{equation}
 g = \gamma_1^+ u = \gamma_1^+ \slp \phi = (-\frac{I}{2} + \dloa) \phi.
\end{equation}
If $\cG^+$ is connected with edge weights $\sigma^+$ we can use
\cref{lem:neu:bie:ext} to conclude that $\phi = (-\frac{I}{2} +
\dloa)^{-1} g$ and thus $u$ admits the unique solution $u = \slp
(-\frac{I}{2} + \dloa)^{-1} g$. This may be surprising at first, but the
Dirichlet boundary condition at the $B$ nodes ensure uniqueness for
this problem.

\section{Active cloaking on graphs}
\label{sec:cloaking}
Consider the problem of hiding an anomaly in a known graph from the
perspective of electrical measurements made away from the anomaly. The
anomaly may consist of a conductivity discrepancy from the known
reference graph, a topological change (e.g. a deleted edge or a new
connection), a current source localized to a few nodes or even some
nodes maintained at a particular potential. We shall use the
reproduction formulas \cref{sec:repro} to hide such anomalies by
injecting or extracting currents at certain nodes surrounding the
anomaly. This \emph{active cloaking} strategy is inspired by similar
techniques for the Laplace and Helmholtz equations
\cite{Miller:2005:OPC, Guevara:2009:AEC, Guevara:2011:ECA,
Guevara:2012:MAT, Guevara:2013:TEA, Norris:2012:SAF, Norris:2014:AEC},
the heat equation \cite{Cassier:2020:ATC, Cassier:2022:AEC}, elastic
waves \cite{Norris:2014:AEC} and flexural waves \cite{ONeill:2015:ACO,
McPhedran:2016:ACO}. One possible motivation for active cloaking on
graphs is for voltage control in an electric distribution grid, where it
is desirable to keep a nominal distribution voltage (e.g. 110V) even in
the presence of a fault. A fault could be, for instance, a new
connection to
the ground (zero potential) or a source (power plant) that is no longer
operational. Of course the weighted graph model we use is too simplistic
to deal with a realistic power grid scenario, since we only consider direct
currents and ignore transient effects.

To fix ideas, consider a known reference graph $\cG$ with nodes 
partitioned as in \cref{sec:setup} and with known conductivity $\sigma$.
We assume that the anomaly is localized in $\Omega^- \setminus
\cN(\partial\Omega)$, so no edge connecting $\Omega^-$ to $\partial
\Omega$ is affected. Our goal is to inject currents in $\partial\Omega$
(and neighboring nodes) to control voltages so that the effect of the
anomaly cannot be detected by means of electrical measurements in the
exterior $\Omega^+$. In \cref{sec:repro} we saw exterior or interior
representation formulas that can control voltages in either $\Omega^+$
or $\Omega^-$. Whether an interior or exterior representation formula is
appropriate depends on the kind of anomaly. We include in the
Supplementary Materials, additional experiments applying this cloaking
method to cloaking anomalies from random walks in the sense of expectation.

\subsection{Cloaking using the interior reproduction formulas}
\label{sec:interior}
Let us start by assuming that the anomaly cannot be detected when
$\partial\Omega$ is set to a constant voltage, say zero. This is the case for
conductivity or topological changes localized to the nodes in
$\Omega^- \setminus \cN(\partial\Omega)$. Indeed, by the maximum
principle the voltages should be constant on $\Omega^-$ as well, so
conductivity or topological changes cannot be detected because there is
no current flow. To put it all together, assume we have a probing field
$u$ that satisfies the Dirichlet problem
\begin{equation}
 (L(\sigma)u)|_{\vint} = 0, ~ u|_B = f,
\end{equation}
for some known $f \in \real^{|B|}$. We can use the interior
representation formula \eqref{eq:intrep+} on $-u$ to obtain a ``cloaking
voltage'' $u_c = - u \ifun_{\Omega^- \cup \partial\Omega}$. By
linearity, the resulting total voltage $u_{\text{tot}} = u + u_c = u
\ifun_{B \cup \Omega^+}$ is indistinguishable from $u$ on $B \cup
\Omega^+$. According to \eqref{eq:intrep+}, this is achieved by
injecting currents on $\cN(\partial\Omega) \setminus \Omega^-$, i.e. since
\begin{equation}
 u_c = -\slp(\gamma_1^+u) + \dlp^+(\gamma_0 u) = G
 \comm{R_{\partial\Omega}^T R_{\partial\Omega} , L(\sigma^+) } u,
\end{equation}
the current density is $\comm{R_{\partial\Omega}^T R_{\partial\Omega} ,
L(\sigma^+) } u$. Similarly cloaking can be achieved by using the other
interior representation formula \eqref{eq:intrep-}, in which case the
currents are injected on $\cN(\partial\Omega) \setminus \Omega^+$. Therefore a
similar cloaking effect can be obtained as long as the anomaly does not
affect such nodes.

To illustrate this technique, we consider the lattice depicted in
\cref{fig:cl_lat}, where the conductivity is constant equal to $1$. The
boundary condition is $0$ at the bottom nodes, $2$ at the top and $1$
at the remaining nodes in $B$. The reference voltage $u$ is shown in
\cref{fig:lat_ex} (left). The voltage $u_*$ with the anomaly is
shown in \cref{fig:lat_ex} (center). Here the anomaly consists of the
node at the center of the lattice having zero voltage. In
\cref{fig:lat_ex} (right) we inject currents using \eqref{eq:intrep+} so
that $u_* + u_c$ is zero around the defective node (i.e. on
$\Omega^-$), without disturbing the voltages in $\Omega^+$.

\begin{figure}
    \centering
    \begin{tabular}{cc}
        \includegraphics[width=0.20\textwidth]{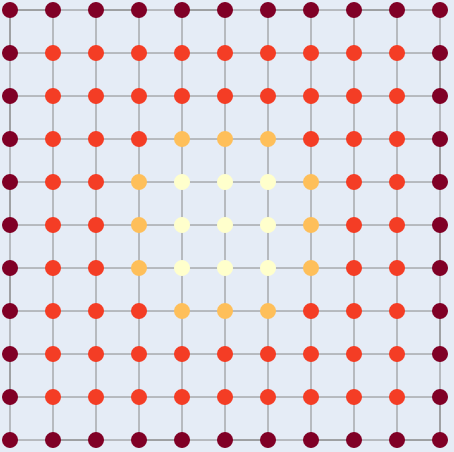} & 
         \includegraphics[width=0.20\textwidth]{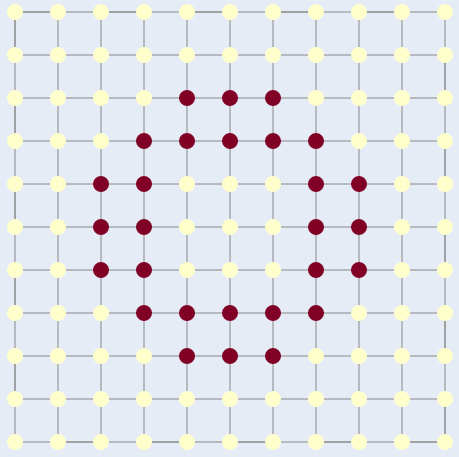}
    \end{tabular}
    \caption{An illustration of the graph setup used for the lattice
    cloaking problem. The node setup of the lattice, from darkest to
    lightest are the sets $B$, $\Omega^+$, $\partial \Omega$, and
    $\Omega^-$ (left). The set $\cN(\partial\Omega) \setminus \Omega^-$
    where currents are injected is shown as shaded nodes (right).} 
    \label{fig:cl_lat}
\end{figure}

\begin{figure}
    \centering
    \begin{tabular}{ccc}
        \includegraphics[width=0.20\textwidth]{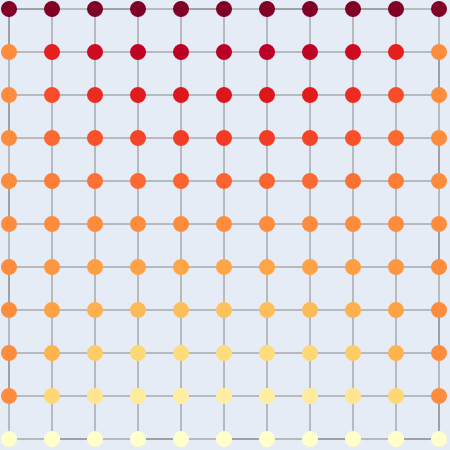} & 
         \includegraphics[width=0.20\textwidth]{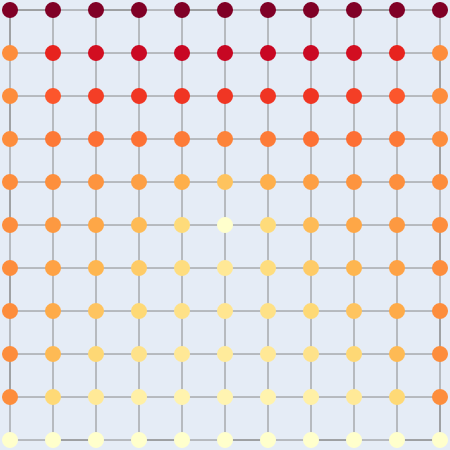} &
         \includegraphics[width=0.20\textwidth]{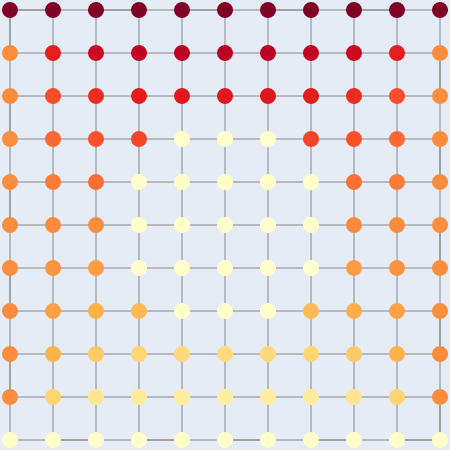} 
    \end{tabular}
    \caption{An example of cloaking a defective node on a lattice. The
    reference solution $u$ (left) is compared to the solution
    $u_*$ with a defective node (center) with zero Dirichlet
    condition (the center node of the lattice). Cloaking of the
    defective node is achieved by using the interior representation
    formula \eqref{eq:intrep+}, creating the voltage $u_* +
    u_c$ (right).} 
    \label{fig:lat_ex}
\end{figure}
\subsection{Cloaking using the exterior reproduction formulas}
\label{sec:exterior} 
Now assume that the anomaly generates currents at some nodes in
$\Omega^-$, then the resulting voltage $u_*$ satisfies
\begin{equation}
 (L(\sigma) u_*)|_{\Omega^+} = 0,~u_*|_B = 0.
\end{equation}
Therefore as long as the anomalous source inside $\Omega^-$ does not
coincide with the nodes $\cN(\partial\Omega)$, we can use either the
exterior reproduction formulas \eqref{eq:extrep+} or \eqref{eq:extrep-}
on $-u_*$ to generate a cloaking field $u_c$ such that the total field
$u_{\text{tot} = u_* + u_c}$ vanishes at the nodes $\Omega^+\cup B$,
preventing detection of the anomaly from the perspective of electric
measurements made at $\Omega^+ \cup B$. The particular exterior
reproduction formula dictates the injection nodes for the currents
needed to generate $u_c$. For formula \eqref{eq:extrep+}, the injection
locations are at the nodes $\cN(\partial\Omega) \setminus \Omega^-$
and give $u_c = -u_* \ifun_{B\cup \Omega^+}$ so that $u_{\text{tot}} =
u_* \ifun_{\partial\Omega \cup \Omega^-}$.  Whereas for
\eqref{eq:extrep-}, the currents are injected at the nodes 
$\cN(\partial\Omega) \setminus \Omega^+$ so that $u_c = -u_*
\ifun_{B\cup \Omega^+ \cup \partial\Omega}$ and $u_{\text{tot}} = u_*
\ifun_{\Omega^-}$. Both approaches are illustrated in \cref{fig:srccl}
in the same lattice graph. This time the anomaly consists in injecting
currents at a single node at the lattice center and we can observe how
the total field vanishes on $\Omega^+$ while being unmodified near the
anomaly.

\begin{figure}[h]
    \centering
    \begin{tabular}{ccc}
        \includegraphics[width=0.20\textwidth]{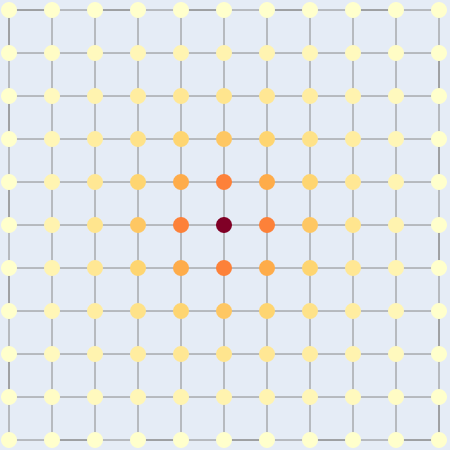} & 
        \includegraphics[width=0.20\textwidth]{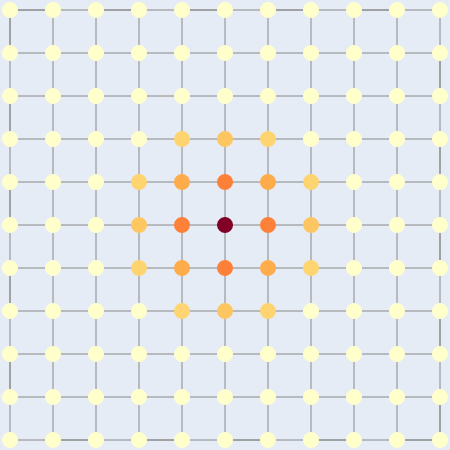} &
         \includegraphics[width=0.20\textwidth]{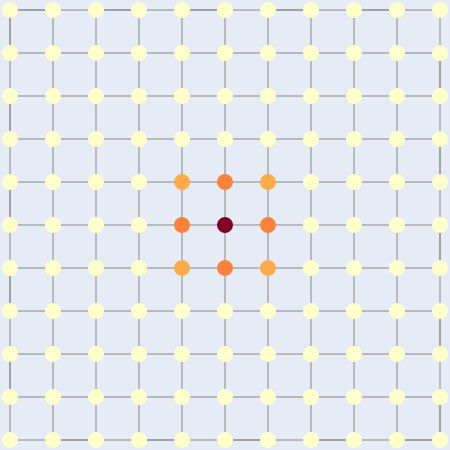} 
    \end{tabular}
    \caption{A numerical example of cloaking a source on a lattice. The
    source field without the cloak (left) is compared to the field with
    the cloak active. In the center, we inject currents at
    $\cN(\partial\Omega) \setminus \Omega^-$ and the right at
    $\cN(\partial\Omega) \setminus \Omega^+$. Notice how the fields in
    the middle and right panes are very close to zero outside of
    $\Omega^-$.}
    \label{fig:srccl}
\end{figure}

\begin{remark}\label{rem:mimicking}
There other variations of this problem that we have not considered. One
is the \emph{mimicking problem}, which consists in making an object or
source in the cloaked region appear as if it were another object or
source from the perspective of measurements made away from the cloaked
region \cite{Lai:2009:IOT}. As shown in \cite{Zheng:2010:EOC} in the
continuum, active cloaking strategies are well adapted to achieve this.
Although not illustrated here, mimicking can also be achieved for
objects or sources inside $\Omega^-$ by cancelling out the probing field
inside $\Omega^-$ and using an exterior reproduction formula to
replicate a field imitating another object or source, as observed in the
exterior $\Omega^+$.
\end{remark}

\section{Summary and perspectives}
\label{sec:summary}
We have constructed a discrete potential theory that relies on a
partition of the Laplacian and that includes the familiar boundary layer
and boundary potential operators from the continuum. Crucially, our
formalism makes only very mild assumptions about the topology, roughly
that nodes are separated by the boundary $\partial\Omega$ into ``interior'' and
``exterior''. With this in place, we can express solutions to different
boundary value problems in terms of ``discrete boundary integral
equations'' that involve the different boundary layer operators that we
defined.  So far we have only proved that a few basic results from
potential theory also hold for weighted graphs, however we believe other
results can be translated. Another
motivation behind this formalism is to bring the language from the
continuum boundary integral equations to numerical methods, aiming at
better mimicking the continuum properties. One drawback in our approach
is the reliance on the Green function, which may be expensive to compute
and store.  However this is a fixed cost and only the Green functions
associated with $\partial\Omega$ are needed. Applying our results to the
discrete equivalent of the Helmholtz equation would be interesting.
Another natural generalization of our approach is to periodic infinite
lattices.  In this case the Green function is often known explicitly and
is location invariant.

\textbf{Data Access.}  All the code needed to reproduce the results in
this paper is available in the repository \cite{thecode} and is written
using the Julia programming language. The figures can be reproduced
using \texttt{RW\_cloaking\_demo.ipynb}. The main discrete potential
theory identities can be verified numerically on a small randomly generated
weighted graph with \texttt{verification.ipynb}. The supplementary material
\texttt{rwvideo.mp4} can be generated using \texttt{rwanim.jl}. Please
see the file \texttt{README.md} in the repository for more details.

\textbf{Acknowledgements.} The authors would like to acknowledge
stimulating discussions with Maxence Cassier, S\'ebastien Guenneau and
Grigorios A. Pavliotis about controlling random walks, which inspired
this work.

\appendix
\section{Supplementary materials}

We recall in \cref{sec:rw} a definition of random walks on graphs, which
differs a bit from the more common definition in that the particles
carry a charge. In some sense this is closer to the physical behavior in
resistor networks, where the particles are electrons. Then in
\cref{sec:rwen} we tie the expected charges at the nodes to
$\sigma-$harmonic functions on graphs. Finally in \cref{sec:acrw} we
report some numerical experiments for active cloaking for random walks.
The material here is for the most part self-contained, but assumes
familiarity with the main document's notation. 

\subsection{Random walks in weighted graphs}
\label{sec:rw}
We consider random walks of charged particles on the graph $\cG$ that
terminate at the boundary $B$ and where the charge is determined by the
random walk's starting node \cite{Georgakopoulos:2011:ENR}. The
transition probability is determined by a conductivity $\sigma \in
(0,\infty)^{|\cE|}$. To be more precise, the random walks can be modeled
by a sequence of random variables $X_0,X_1,X_2,\ldots$ with
\begin{equation}
    \proba(X_0 = x_0, X_1 = x_1,\ldots,X_n = x_n) = 
    \pi(x_0) p(x_0,x_1) p(x_1,x_2) \cdots p(x_{n-1},x_n),
\end{equation}
where $\pi(x)$ is the initial particle probability distribution and
$p(x,y)$ is the probability that a particle at $x \in \cV$ goes
to $y \in \cV$, which for $x$ and $y$  connected by an edge $\{x
, y\} \in \cE$ is given by
\begin{equation}\label{eqn:tran_prob}
    p(x,y) = \sigma(\{x,y\})/c(x),
    ~\text{where}~
    c(x) = \sum_{z~\text{s.t.}~\{x,z\}\in\cE} 
              \sigma(\{x ,z\}),~\text{for}~x\in \cV,
\end{equation}
and $p(x,y) = 0$ if there is no edge between $x$ and $y$. The definition
of the transition probabilities \eqref{eqn:tran_prob} results in a
reversible Markov chain, see e.g. \cite{Lyons:2016:PTN}.
The first time that a particle hits the boundary is
\begin{equation}
    \tau_B = \inf \{ n\geq 1 ~|~ X_n \in B \}.
\end{equation}
We assume that the initial particle distribution is uniform on a subset
$\cU$ of the vertices $\cV$ i.e.
\begin{equation}
 \pi(x) = \frac{1}{|\cU|}\ifun_{\cU}(x),
\end{equation}
where $\ifun_{\cU}$ is the indicator function of the set
$\cU$ (equal to one for $x \in \cU$ and zero otherwise)
and $|\cU|$ is the cardinality of $\cU$. Let $\hat{q}
\in \real^{|\cU|}$ be the charges that are assigned to particles
departing from $\cU$, and extend $\hat{q}$ by zero outside of
$\cU$. The \emph{expected charge} that a node $x$ encounters is
defined by $q \in \real^{|\cV|}$ with 
\begin{equation}
q(x) = |\cU| \, \expect \left[ 
  \hat{q}(X_0) \sum_{i=0}^{\tau_B - 1} \ifun_{\{ X_i = x\}} 
  \right].
\label{eq:q}
\end{equation}
We emphasize that in our setup the random walk starting vertices can
potentially be on any node and are not restricted to $B$ as in
\cite{Georgakopoulos:2011:ENR}.
\begin{lemma}
\label{lem:echarge}
The expected charge  satisfies
\begin{equation}
q(x) = 
 \begin{cases}
    \hat{q}(x), 
    &\text{if}~x\in B,\\
    \displaystyle
    \hat{q}(x) + 
    \sum_{y~\text{s.t.}~\{x,y\}\in\cE} p(y,x)q(y), 
    &\text{if}~x\in \vint= \cV \setminus B.
 \end{cases}
\end{equation}
\end{lemma}
\begin{proof}
Let $x\in B$. Since random walks terminate at the boundary, the only
time that a random walk may encounter an $x \in B$ is $i=0$, thus the
sum in \eqref{eq:q} reduces to one term
\begin{equation}
    q(x) = |\cU| \, \expect \left[ \hat{q}(x) \ifun_{\{ X_0 = x\} }\right] 
    = |\cU| \hat{q}(x) \pi(x) = \hat{q}(x) \ifun_{\cU} (x),
    ~\text{for}~x\in B.
\end{equation}
For other vertices $x \in \vint$, we have
\[
\begin{aligned}
q(x) &= |\cU| \expect \left[ \hat{q}(x) \ifun_{\{X_0=x\}} \right] 
      + |\cU| \expect \left[ \hat{q}(X_0) 
      \sum_{i=1}^{\tau_B - 1}\ifun_{\{ X_i = x\}} \right]\\
&= |\cU| \hat{q}(x) \pi(x)  
 + |\cU|\expect \left[ \hat{q}(X_0) \sum_{i=1}^{\tau_B - 1} 
   \sum_{y~\text{s.t.}~\{x,y\}\in\cE} \ifun_{\{ X_{i-1} = y, X_i = x\}}\right]\\
&= \hat{q}(x) 
+\sum_{y~\text{s.t.}~\{x,y\}\in\cE} p(y,x)|\cU|
  \expect \left[ \hat{q}(X_0) \sum_{i=1}^{\tau_B - 1} \ifun_{\{ X_{i-1} = y\}} \right].
\end{aligned}
\]
The equality is unaffected if in the last summation we take the upper
bound for $i$ to be $\tau_B$ instead of $\tau_B -1$. Indeed, $y$ is the
particle location before $x$ and thus cannot be on the boundary. By
changing indices to $j=i-1$ in the summation, we recognize $q(y)$ and
obtain the desired result 
\[
q(x)=\hat{q}(x) + \sum_{y~\text{s.t.}~\{x,y\}\in\cE} p(y,x)q(y).
\]
\end{proof}

\subsection{Random walks and electrical networks}
\label{sec:rwen}
Define the function
\begin{equation}\label{eqn:charges}
u = q/c \in \real^{|\cV|},
\end{equation}
where the division is understood componentwise. By using
\cref{lem:echarge} and 
\begin{equation}
 \sum_{y~\text{s.t.}~\{x,y\}\in\cE} p(x,y) = 1,
  ~\text{for a fixed}~ x \in \cV,
\end{equation}
we see that $u$ satisfies
\begin{equation}
 \label{eq:kirchhoff}
 \begin{aligned}
    \sum_{y~\text{s.t.}~\{x,y\}\in\cE} \sigma(\{x,y\}) (u(x) - u(y)) 
    &= \hat{q}(x),
    &&~\text{for}~x \in \cV \setminus B,\\[-1em]
    u(x) & = \hat{q}(x) / c(x),&&~\text{for}~x \in B.
 \end{aligned}
\end{equation}
Thus $u$ can be interpreted as the voltages at the nodes of a resistor
network with edge conductances given by $\sigma \in
(0,\infty)^{|\cE|}$, where the voltages at the boundary nodes 
are $\hat{q}|_B/c|_B$ and currents of $\hat{q}|_{\vint}$ are injected at
all other nodes. It  is convenient to rewrite \eqref{eq:kirchhoff} in
terms of the weighted graph Laplacian $L(\sigma) = \nabla^T
\diag(\sigma) \nabla$, see \cref{sec:gl} in the main document for
the notation. Let $f \in \real^{|B|}$.  Rewriting \eqref{eq:kirchhoff}
with the graph Laplacian we obtain
\begin{equation}
\label{eq:lap:prob}
\begin{split}
(L(\sigma)u)|_{\vint} &= \hat{q}|_{\vint}, \\
u|_B &= \hat{q}|_B/c|_B.
\end{split}
\end{equation}
Since $L(\sigma)u$ calculates the net currents at all the nodes required
to maintain the potential $u$, the first equation in \eqref{eq:lap:prob}
is a current conservation law: the sum of all the currents leaving a
node $x\in \vint$ is equal to the injected currents $\hat{q}|_{\vint}$.
The second equation in \eqref{eq:lap:prob} corresponds to fixing the
voltages of the boundary nodes to $\hat{q}|_B/c|_B$. If there are no
sources or sinks of particles in $\vint$ i.e.
\begin{equation}\label{eqn:dir_prob}
\begin{split}
(L(\sigma)u)|_{\vint} &= 0, \\
u|_B &= f,
\end{split}
\end{equation}
we obtain the \emph{Dirichlet problem} with (Dirichlet) boundary
condition $f \in \real^{|B|}$. Since $\cG$ is connected, the Dirichlet
problem admits a unique solution for all $f \in \real^{|B|}$, see e.g.
\cite{Curtis:2000:IPE,Chung:1997:SGT}.  Problem \eqref{eq:lap:prob} is
also guaranteed by connectivity to admit a unique solution for any
$\hat{q} \in \real^{|\cV|}$.  We note that the solution to
\eqref{eq:lap:prob} (rescaled by $c$, see \eqref{eqn:charges}) with
$\hat{q}|_B = 0$ is the expected charges (rescaled by $c$, see
\eqref{eqn:charges}) resulting from random walks starting at
$\mathcal{U} = \vint$ and terminating at $B$, with charges at $\vint$
given by $\hat{q}|_{\vint}$.

A numerical example illustrating the relation between the random walk
process with charged particles and the Dirichlet problem
\eqref{eqn:dir_prob} is given in \cref{fig:dir_graph}. In this example,
the edge conductivities are chosen at random from $U(1,2)$ and the
boundary values are chosen at random from $U(0,1)$, where $U(a,b)$ is a
uniform distribution with minimum value $a$ and maximum value $b$. The
solution obtained using the empirical net charges is denoted by
$\tilde{u}$ and is calculated using $10^5$ realizations. The relative
error $E(u,\tilde{u};\cV) \approx 0.47 \%$ where $u$ is the
deterministic solution obtained by solving \eqref{eqn:dir_prob}. The
relative error between $f,g \in \real^{\cV}$, restricted to some set $A
\subset \cV$ is
\begin{equation}\label{eqn:rel_err}
    E(f,g;A)= \frac{\|(f-g)|_A\|_2}{\|f|_A\|_2},~\text{for}~f\neq 0.
\end{equation}

\begin{figure}[h]
\centering
\begin{tabular}{c}
    \includegraphics[width=\textwidth]{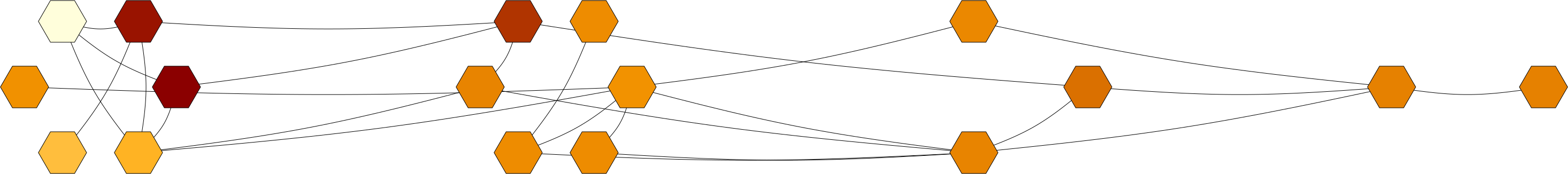}  \\
     \includegraphics[width=\textwidth]{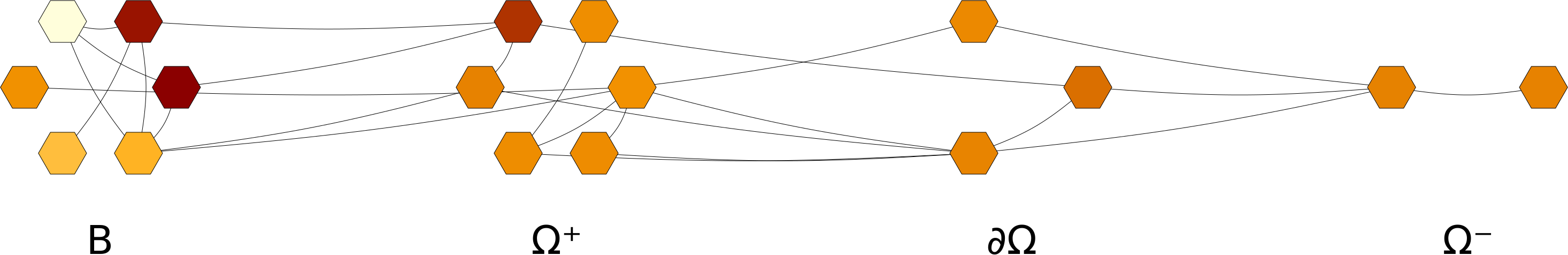}
\end{tabular}
    \caption{Solution of the Dirichlet problem (top) and random walk
    solution given by \eqref{eqn:charges} (bottom). Darker shades
    indicate higher values and lighter shades indicate smaller values.} 
    \label{fig:dir_graph}
\end{figure}

\begin{remark}\label{rem:interpretation}
A more common relation between random walks in graphs and the Dirichlet
problem is given in  \cite{Doyle:1984:RWE,Lyons:2016:PTN}, where the
random walks start at nodes in $\vint$ and terminate at the boundary
nodes $B$. When the random walk terminates at $x \in B$, the payoff
value (which may be negative) is given by $f(x)$, where $f \in
\real^{|B|}$. Then, the expected payoff for a random walk starting at
any node $x$ is given by $u(x)$ solving the Dirichlet problem
\eqref{eqn:dir_prob} with $u|_B = f$. We chose instead the ``charged
particles'' point of view in \cite{Georgakopoulos:2011:ENR}, because it
allows us to restrict the random walk starting nodes (or particle
injection locations) to the boundary $B$ or to another node subset
of $\cV$.
\end{remark}

\subsection{Active cloaking for random walks}
\label{sec:acrw}
We now apply the active cloaking strategy from \cref{sec:cloaking} to
random walks. Indeed the currents that need to be injected at nodes
in $\cN(\partial\Omega)$ can be interpreted as the charges of particles
injected at $\cN(\partial\Omega)$.  To see this, let us first consider a
solution $u$ to the Dirichlet problem \eqref{eqn:dir_prob} with $u|_B =
f$. In terms of random walks, $q = cu$ are the expected charges
resulting from charged particles doing random walks starting at $B$ with
charges $c|_B f$.  By applying the interior representation formula
\eqref{eq:intrep+} to $-u$, we can define a charge distribution for
injected particles on $\cN(\partial\Omega) \setminus \Omega^-$ given by
$\phi_c^+ = \comm{R_{\partial\Omega}^T R_{\partial\Omega}, L(\sigma^+)}$
so that the resulting expected charge distribution is $q_c^+ =
-cu\ifun_{\partial\Omega \cup \Omega^-}$. By linearity the expected
charge distribution created by the particles injected at
$\cN(\partial\Omega) \setminus \Omega^-$ is then $q_{\text{tot}}^+  = c
u_{\text{tot}}^+ =  c u \ifun_{B \cup \Omega^+}$. This is illustrated
in \cref{fig:cl_graph} (top) for the $u$ reported in
\cref{fig:dir_graph}. We use tildes to denote quantities that are
obtained empirically. In this case with $10^5$ random walks, the error
was  $E(u^+_c,\tilde{u}^+_c, \cV)\approx0.47\%$, calculated
using \eqref{eqn:rel_err}.

If we instead apply \eqref{eq:intrep-} on $-u$, the charge distribution
for injected particles at $\cN(\partial\Omega) \setminus \Omega^+$ is
given by $\phi_c^- = -\comm{R_{\partial\Omega}^T R_{\partial\Omega},
L(\sigma^-)}$. The resulting total charge distribution is
$q_{\text{tot}}^-  = c u_{\text{tot}}^- = c u \ifun_{B \cup \Omega^+
\cup \partial\Omega}$.  As before, we calculated an empirical
$\tilde{u}_{\text{tot}}$ for $10^5$ random walks and report it in
\cref{fig:cl_graph} (bottom). The error was $E(u^-_c,\tilde{u}^-_c,
\cV)\approx1.76\%$.

\begin{figure}[h]
\centering
\begin{tabular}{cc}
    \includegraphics[width=\textwidth]{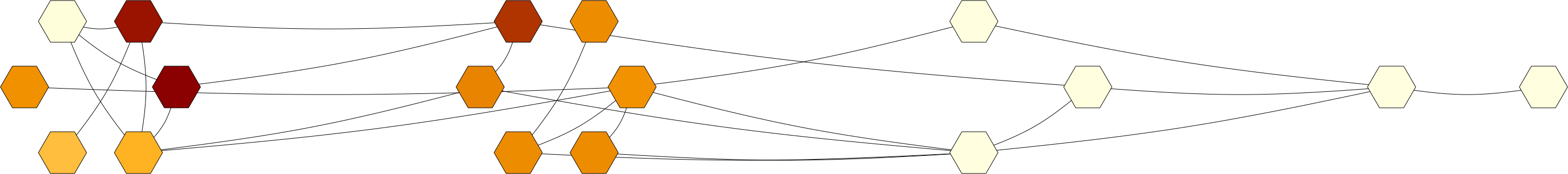}  \\
     \includegraphics[width=\textwidth]{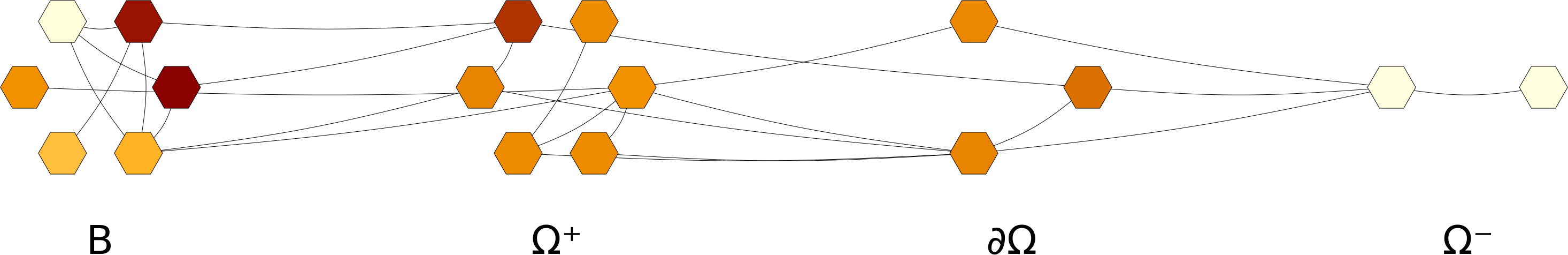}
\end{tabular}
    \caption{Total fields in the active cloaking problems using the
    interior reproduction formula \eqref{eq:intrep+} (top) and 
    \eqref{eq:intrep-} (bottom). The empirical total fields 
    $\tilde{u}_{\text{tot}}^{\pm}$ were calculated using $10^5$
    random walks.} 
    \label{fig:cl_graph}
\end{figure}

We also simulated the cloaking examples in a lattice from the main text
using random walks (\cref{fig:lat_ex,fig:srccl}). For the interior
problem (\cref{fig:lat_ex}), we have $E(u_c,\tilde{u}_c,
\cV)\approx0.38\%$ with $10^5$ random walks. Similarly in 
\cref{fig:srccl} we have $E(u_\text{tot}^+,\tilde{u}_\text{tot}^+,
\cV)\approx0.46\%$ and $E(u_\text{tot}^-,\tilde{u}_\text{tot}^-,
\cV)\approx0.73\%$. 

We also include the animation \texttt{rwvideo.mp4}, which represents a
study done on a $20 \times 20$ lattice with edges having the same
conductivity. This study compares the empirical expected net charges
stemming from: (a) particles originating at the boundary $B$ (the
uncloaked configuration), (b) particles originating at an active set
$\cN(\partial\Omega)\setminus \Omega^-$ (the cloak by itself) and (c) particles originating at both $B$
and $\cN(\partial\Omega)\setminus \Omega^-$ (the cloaked configuration). Rather than calculating the
expected net charges when all the random walks are finished, we
periodically start batches of random walkers. As expected (b) shows
charges approaching the negative of those in (a) inside the cloaked
region and zero outside and (c) shows net charges approaching those in
(a) outside of the cloaked region and zero inside.

\bibliographystyle{elsarticle-harv}
\bibliography{dpt}
\end{document}